\documentclass[10pt]{amsart}

\usepackage{amsfonts,amsmath,latexsym,amssymb,verbatim,amsbsy,amsthm}
\usepackage{dsfont}

\usepackage{graphicx}

\usepackage{caption}
\usepackage{float}
\usepackage{wrapfig}

\usepackage[top=0.5in, bottom=0.5in, left=0.5in, right=0.5in]{geometry}

\usepackage[dvipsnames]{xcolor}

\usepackage[colorlinks=true, pdfstartview=FitV, linkcolor=RoyalBlue,citecolor=ForestGreen, urlcolor=blue]{hyperref}

\theoremstyle{plain}
\newtheorem{THEOREM}{Theorem}[section]

\newtheorem{theorem}[THEOREM]{Theorem}
\newtheorem{corollary}[THEOREM]{Corollary}
\newtheorem{lemma}[THEOREM]{Lemma}
\newtheorem{proposition}[THEOREM]{Proposition}
\newtheorem{assumption}[THEOREM]{Assumption}

\theoremstyle{definition}

\theoremstyle{remark}

\newtheorem{remark}[THEOREM]{Remark}

\newtheorem{example}[THEOREM]{Example}

\def \e {\varepsilon}

\def \l {\lambda}

\def \t {\tau}
\def \th {\theta}

\def \L {\Lambda}

\def \O {\Omega}

\def \cD {\mathcal{D}}

\newcommand{\R}{\ensuremath{\mathbb{R}}}   
\newcommand{\C}{\ensuremath{\mathbb{C}}}   

\def \dx  {\, \mbox{d}x}
\def \dt  {\, \mbox{d}t}

\def \dr  {\, \mbox{d}r}
\def \ds  {\, \mbox{d}s}

\def \dtau  {\, \mbox{d}\tau}

\def \ddt  {\frac{\mbox{d\,\,}}{\mbox{d}t}}

\def \Re {\mathrm{Re}}
\def \Im {\mathrm{Im}}

\title{Schr\"odinger-Lohe type models of quantum synchronization with nonidentical oscillators}
\author[P. Antonelli]{Paolo Antonelli}
\address{Gran Sasso Science Institute, viale Francesco Crispi, 7, 67100 L'Aquila, Italy}
\email{paolo.antonelli@gssi.it}

\author[D. Reynolds]{David N. Reynolds}
\address{Gran Sasso Science Institute, viale Francesco Crispi, 7, 67100 L'Aquila, Italy}
\email{david.reynolds@gssi.it}

\date{\today}

\begin{document}

\maketitle
\begin{abstract}
 We study the asymptotic emergent dynamics of two models that can be thought of as extensions of the well known Schr\"odinger-Lohe model for quantum synchronization. More precisely, the interaction strength between different oscillators is determined by intrinsic parameters, following Cucker-Smale communication protocol. Unlike the original Schr\"odinger-Lohe system, where the interaction strength was assumed to be uniform, in the cases under our consideration the total mass of each quantum oscillator is allowed to vary in time. A striking consequence of this property is that these extended models yield configurations exhibiting phase, but not space, synchronization. The results are mainly based on the analysis of the ODE systems arising from the correlations, control over the well known Cucker-Smale dynamics, and the dynamics satisfied by the quantum order parameter.
 \\
 
\end{abstract}

\section{Introduction}
The study of synchronization behavior has been studied for the last 400 years since Huygens noticed the synchronization of two pendulum clocks coupled by hanging on the same beam. The same synchronous behavior can also be seen in the synchronization of pacemaker cells in hearts, descriptions of social phenomena like a crowd clapping, and biological studies of crickets and fireflies. The quintessential model for describing and analyzing synchronization behavior is the Kuramoto model.

The classical Kuramoto model \cite{KUR} consisting of $N$ coupled phase oscillators and randomly distributed natural frequencies has been a subject of much study in the last half century since its inception in 1975, see for instance the review \cite{ABPRS}. Its deceptively simple design serves to model these classical synchronization behaviors impressively well. 
The natural question that Lohe \cite{LH2} addressed was whether we could observe similar synchronization behavior in higher-dimensional generalizations of the Kuramoto model, where the oscillators can move on a $n-$dimensional sphere. The models in \cite{LH2} can also be interpreted as a system of Schr\"odinger equations for wave functions attaining only $n$ quantum energy levels, see also \cite{LH1} where this quantum mechanical formulation was put forward in the infinite dimensional case, namely by considering the wave functions to belong to the Hilbert space $L^2(\R^d)$.

The standard Schr\"odinger-Lohe model \cite{LH1} may be written as
\begin{align}\label{eq:SL_or}
i\partial_t \psi_j&=H_j\psi_j+\frac{ik}{2N}\sum_{l=1}^N\left(\psi_l-\frac{\langle \psi_l,\psi_j\rangle}{\|\psi_j\|_2^2} \psi_j\right),\\
\psi_j(x,0)&=\psi_j^0(x), \ \ j=1,\dots,N, \nonumber
\end{align}
with coupling strength $k>0$ and $\langle f,g \rangle=\int_{\R^d} \bar{f}(x)g(x) \dx$ is the standard inner product on $L^2(\R^d)$, and $H_j$ are Hamiltonian operators, that we assume to have the form $H_j=H+\O_j$, where the real constants $\O_j$ play the same role of the natural frequencies in the classical Kuramoto model. Our results apply to a large class of Hamiltonian operators $H$ that will be discussed below.

The Schr\"odinger-Lohe model describes a system of $N$ quantum oscillators, all connected to each other, nonlinearly interacting among themselves through unitary transformations. In particular this means that the $L^2$ norm of each wave function is conserved in time (as it can also be checked by direct calculation), despite the presence of the non-Hamiltonian terms in \eqref{eq:SL_or} that drive the system towards synchronization.

Although system \eqref{eq:SL_or} was not derived as a relevant model for specific applications, it is a relatively simple generalization of the Kuramoto system to a quantum context, that serves as a good theoretical starting point to investigate synchronization in quantum systems. 
In fact, the modeling of some quantum circuits suggests that nonlinearities in connecting optical fibers could be described by the nonlinear effects encoded in \eqref{eq:SL_or}, see \cite{LNNTB, Nigg} for related discussions.
Moreover, the study of quantum synchronization models as in \eqref{eq:SL_or} could provide some insights in the study of quantum coherence and entanglement, with applications towards quantum computations, communication and metrology \cite{Kim}.

The Schr\"odinger-Lohe model \eqref{eq:SL_or} and the asymptotic behavior of its solutions was rigorously studied in \cite{AM, HHK}, see also \cite{CCH, HK_frust}. In particular, the case of two oscillators is well understood, as it is possible to show complete frequency synchronization under general assumptions on the initial data. On the other hand, the general system \eqref{eq:SL_or} with $N\geq3$ presents major difficulties so that a result is available only for identical oscillators, i.e. $\Omega_j=0$, for all $j=1, ..., N$, see also \cite{HHK}. In this case, under some assumptions on the initial data, the system experiences complete phase synchronization. 

Let us remark that, by formally considering wave functions with constant amplitude and that depend only on time, say $\psi_j=e^{iS_j}$, with $S_j=S_j(t)$, then the classical Kuramoto model is recovered from \eqref{eq:SL_or}. Moreover, the fact that in model \eqref{eq:SL_or} the individual $L^2-$norms are conserved, implies that not only the wave functions synchronize their phases, but they actually converge to the same quantum state.

Our aim in this paper is to investigate situations where the 
strength of each oscillator's ability to synchronize varies in time.
For this purpose, we introduce time dependent parameters $\theta_j$, that steer the Schr\"odinger-Lohe dynamics. In the present analysis such parameters control the coupling strength of the oscillators, see also Remark \ref{rmk:theta} below, and they are considered as instrinsic, determined by a Cucker-Smale protocol. In this way there may be configurations, that will be thoroughly discussed in our paper, where phase synchronization may occur without space synchronization. Indeed, a crucial feature provided by the introduction of the parameters $\theta_j$ is that the individual $L^2$ norms of each $\psi_j$ are no longer conserved.
\newline
The first model we consider reads as follows
\begin{align}\label{model1}
i\partial_t \psi_j&=H_j\psi_j+\frac{ik}{2N}\sum_{l=1}^N\left(\th_j\psi_l-\frac{\langle \psi_l,\psi_j\rangle}{\|\psi_j\|_2^2} \psi_j\right),\nonumber\\
\psi_j(x,0)&=\psi_j^0(x), \ \ j=1,\dots,N,\\
\dot{\th_j}&=\frac{\mu}{N}\sum_{k=1}^Nh(|x_j-x_k|)(\th_k-\th_j), \ \ \mu > 0, \nonumber\\
x_j&=\frac{1}{\|\psi_j\|_2^2}\int_{\mathbb{R}^d} x|\psi_j(x,t)|^2 \dx, \nonumber
\end{align}
where $h(r)$ is a radially symmetric influence function, and $x_j=x_j(t)$ represents the center of mass of the oscillator $\psi_j$. We will refer to this as Model 1. The second variation of the Schr\"odinger-Lohe model, or Model 2, we consider is given by
\begin{align}\label{m2}
i\partial_t \psi_j&=H_j\psi_j+\frac{ik}{2N}\sum_{l=1}^N\left(\th_j\psi_l-\langle \psi_l,\psi_j\rangle \psi_j\right),\nonumber\\
\psi_j(x,0)&=\psi_j^0(x), \ \ j=1,\dots,N,\\
\dot{\th_j}&=\frac{\mu}{N}\sum_{k=1}^Nh(|x_j-x_k|)(\th_k-\th_j), \ \ \mu \geq 0, \nonumber\\
x_j&=\frac{1}{\|\psi_j\|_2^2}\int_{\mathbb{R}^d} x|\psi_j(x,t)|^2 \dx. \nonumber
\end{align}
We notice that the only difference between Model 1 and Model 2 is the inclusion of the normalization by the $L^2$ norm in the synchronization part of the equation. This small difference leads to considerable dissimilarities in what results we are able to prove, and consequently the nature of each of the respective models.
Let us remark that similar coupling, involving parameters driven by the Cucker-Smale protocol has been recently considered in some alignment models of collective behavior \cite{LRS}, \cite{NR}, but not yet in the context of (classical and/or quantum) synchronization.

The non-conservation of the individual $L^2-$norms allows configurations where phase synchronization is achieved without space synchronization. In Theorems \ref{t:synch1} and \ref{ncssynch} we will show some results where phase synchronization occurs, but the $L^2-$norms of the single wave functions will be different. Further we give explicit examples \ref{bip}, and \ref{inc} of the emergence of both an unstable Bipolar synchronization state, and an incoherent state, respectively.

We conclude this introduction by mentioning some further results already existing in the literature about the Schr\"odinger-Lohe model \eqref{eq:SL_or} and related variants. In \cite{CCH} practical synchronization (i.e. synchronization in the limit when the coupling constant $k$ becomes larger and larger) was proved, \cite{AHKM} deals with the Wigner-type formulation of system \eqref{eq:SL_or}, \cite{HKR} studies the finite dimensional Lohe model first introduced in \cite{LH2}, \cite{FH} considers the Schr\"odinger-Lohe model with a multiplicative noise, in \cite{HK_frust} the Schr\"odinger-Lohe model with frustration is studied. The paper \cite{HHK_disc} analyzes a discrete version of \eqref{eq:SL_or}. We also address the interested reader to the recent review \cite{HK_rev} for a more comprehensive account of existing results.

The outline of the paper is as follows. 
Section \ref{sec:prel} provides some preliminaries on known models and introduces first basic facts on Models 1 and 2. 
In Section \ref{sec:Cau} we briefly discuss the Cauchy problems associated to systems \eqref{model1} and \eqref{m2}. 
Section \ref{sec:two} is devoted to the analysis of the case of $N=2$ oscillators. 
Complete synchronization for identical oscillators is proved in Theorem \ref{synch2} for both models, while the non-identical oscillator case is dealt with in Theorem \ref{thm:nio} only for Model 1. 
Section \ref{sec:m1} focuses on Model 1 with arbitrary $N$ identical oscillators. 
We prove complete synchronization both in the case of an absolute kernel for the Cucker-Smale protocol, Theorem \ref{t:synch1}, and in the case of $h(r)$ satisfying the Heavy Tail condition, Theorem \ref{thm:synch3}.
Section \ref{sec:m2} focuses on Model 2. 
In Proposition \ref{mod2sync} a wedge condition on the initial data and the Heavy Tail condition on $h(r)$ yields complete synchronization. 
Moreover, Proposition \ref{ncssynch} analyzes the case when the parameters $\theta_j$ are constant in time, i.e. $\mu=0$ in \eqref{m2}. 
Finally in Section \ref{sec:ex} we provide two explicit examples that highlight the differences between \eqref{eq:SL_or} and the models considered here.

\subsection{Notation}
Throughout the paper many constants $C$ appear and will often change from line to line. The inner product $\langle f,g\rangle=\int_{\mathbb{R^d}}\bar{f}(x)g(x) \dx$ will refer to the standard $L^2(\mathbb{R}^d)$ inner product, and $\|\cdot\|_2$ will be the corresponding norm on $L^2(\R^d)$.

\section{Preliminaries}\label{sec:prel}
\subsection{Alignment Preliminaries}
The recent book by Shvydkoy \cite{Rbook} overviews the current state of the art of flocking and alignment models of collective dynamics. However, to start let us lay out some of the definitions and theory that we are using with relation to the Cucker-Smale dynamics. Throughout the paper the parameters $\th_j$, which serve to change the mass of each oscillator (in differing ways depending on which model we are referring to), are controlled by the following Cucker-Smale protocol
\begin{align}
\dot{\th_j}&=\frac{\mu}{N}\sum_{k=1}^Nh(|x_j-x_k|)(\th_k-\th_j), \ \ \mu \geq 0, \nonumber\\
x_j&=\frac{1}{\|\psi_j\|_2^2}\int_{\mathbb{R}^d} x|\psi_j(x,t)|^2 \dx. \nonumber
\end{align}
Note that $\mu=0$ turns off the Cucker-Smale protocol, effectively freezing the parameters $\th_j$. The communication kernel, $h(r)$ will take two different forms throughout the paper:
\begin{enumerate}
\item Absolute kernel: $\inf_{r\geq0} h(r)=c>0$,
\item Heavy Tail: $\int^{\infty} h(r)=\infty$.
\end{enumerate}
Note that absolute kernels satisfy the heavy tail condition, and for either kernel, the symmetry of the kernel leads to the average, $\bar{\th}=\frac{1}{N}\sum_{j=1}^N\th_j$, being preserved in time, $\ddt \bar{\th}=0$, and thus we will let $\bar{\th}=1$. Further an important result that we will use is that regardless of initial conditions, the absolute kernel yields:
\begin{equation}\label{align}
\max_{j=1,...,N}|\th_j-\bar{\th}|(t)\lesssim e^{-ct}
\end{equation}
while for the heavy tail kernel, a priori, we only have
\begin{align}
    \ddt \max_{j=1,...,N}|\th_j-\bar{\th}|(t)\lesssim -\mu h(\cD(t)) \max_{j=1,...,N}|\th_j-\bar{\th}|
\end{align}
where $\cD(t)=\max_{i,j}|x_i-x_j|(t)$. As $x_j$ represents the center of mass of oscillator $\psi_j$, it is necessary to first prove that the distance between centers of masses remains bounded for all time before achieving the alignment of the parameters $\th_j \to \bar{\th}=1$. However, once that is in hand the rate of alignment will be exponential just as in the absolute kernel case. For further background on alignment models we again refer the reader to \cite{Rbook}.

\subsection{Schr\"odinger-Lohe Preliminaries}
The original Schr\"odinger-Lohe model \eqref{eq:SL_or}, with identical natural frequencies, $\O_j=\O$ for all $j=1,...,N$,
has been studied extensively, see \cite{AM,CCH,HHK} and the references therein. As for the Kuramoto model, 
it is possible to define an order parameter, given by
\begin{equation}\label{eq:op}
    \zeta=\frac{1}{N}\sum_{j=1}^N\psi_j,
\end{equation}
so that complete synchronization occurs when $\|\zeta\|_2 \to 1$, while $\|\zeta\|_2=0$ represents a symmetric configuration where each $\psi_j$ evolves independently, according to the Hamiltonian operator $H$. Utilizing the order parameter allows for \eqref{eq:SL_or} to be rewritten in a more compact form,
\begin{align}
    i\partial_t \psi_j&=H_j\psi_j+\frac{ik}{2}\left(\zeta-\frac{\langle \zeta,\psi_j\rangle}{\|\psi_j\|_2^2} \psi_j\right).
\end{align}
Further, the order parameter satisfies its own equation.
\begin{align}
    i\partial_t \zeta&=H\zeta+\frac{1}{N}\sum_{j=1}^N\O_j\psi_j+\frac{ik}{2}\left(\zeta-\sum_{j=1}^N\frac{\langle \zeta,\psi_j\rangle}{\|\psi_j\|_2^2} \psi_j\right).
\end{align}
As shown in \cite{HHK}, when all natural frequencies are identical, the $L^2$-norm of the order parameter evolves monotonically. This may be exploited to asymptotically achieve either an unstable bipolar synchronization, or an exponentially stable complete synchronization. The paper \cite{AM}  treats the system with non-identical natural frequencies in the special case of two oscillators. Depending on the relative size of the natural frequencies, $\O=\frac{1}{2}|\O_1-\O_2|$ and the coupling constant $k$, they achieve complete synchronization $\frac{2\O}{k}\leq 1$ or periodic behavior $\frac{2\O}{k}>1$.  Analysis of the ODE system of correlations is what leads to various synchronization results. For more background on the Schr\"odinger-Lohe system we again refer to \cite{AM,CCH,HHK,LH1}.

\subsection{The Models}
As said above, the purpose of our paper is to investigate two variants of the Schr\"odinger-Lohe system, given by Models 1 and 2, see systems \eqref{model1} and \eqref{m2}, respectively. In both Model 1 and Model 2 the masses of each oscillator no longer remains conserved, indeed for Model 1, we get the following equation for the evolution of the masses,
\begin{equation}\label{e:psi2}
\ddt \|\psi_j\|_2^2=k\Re\langle \zeta, \psi_j\rangle(\th_j-1),
\end{equation}
where again $\zeta=\frac{1}{N}\sum_{j=1}^N\psi_j$ is the order parameter defined in \eqref{eq:op}.
From \eqref{e:psi2} we see that it is necessary that $\th_j \to 1$ in order for $\|\psi_j\|_2^2$ to remain finite through any synchronization process.
This is the reason why $\mu>0$ is necessary for Model 1 in order to achieve synchronization, as otherwise some states could indefinitely grow or vanish, depending on the sign of $\theta_j-1$. On the other hand, Model 2 allows for $\mu=0$, see for instance Proposition \ref{ncssynch} below where synchronization occurs also in this case.

For Model 2, we see a different behavior,
\begin{align}
\ddt \|\psi_j\|_2^2=k\Re\langle \zeta, \psi_j \rangle(\th_j-\|\psi_j\|_2^2),
\end{align}
which shows that, whenever synchronization occurs, $\|\psi_j\|_2^2 \to \th_j$, for each $j=1, \dotsc, N$. 
This implies that the parameter $\theta_j$ steers the $j-$th oscillator to reach the desired mass for arbitrarily large times. 
Overall, we see that the inclusion (or lack thereof) of the normalization in these Schr\"odinger-Lohe type models truly leads to drastically different dynamics for the respective models. For Model 2 we can even provide an example of initial conditions that lead to decay of the order parameter, $\|\zeta\|_2  \to  0$, for instance see the example \ref{inc} in Section \ref{sec:ex} below.

For later purposes, it is convenient to decompose the dynamics of each oscillator $\psi_j$ by separately considering the evolution of its size (i.e. its $L^2$ norm), and a renormalized oscillator.
More precisely, we define
\begin{equation*}\label{eq:size}
\lambda_j(t)=\|\psi_j(t)\|_{2},
\end{equation*}
\begin{equation}\label{eq:ren}
\psi_j(t, x)=\lambda_j(t)\phi_j(t, x).
\end{equation}
In this way, the dynamics for the $j-$th oscillator $\psi_j$ can be decomposed as 
\begin{equation}\label{eq:m1_dec}
\begin{aligned}
\dot\lambda_j=&\,\frac{k}{2}\Re\langle \zeta, \phi_j\rangle(\th_j-1),\\
i\partial_t\phi_j=&\,\left(H_j+\frac{k}{2\lambda_j}\Im\langle\zeta,\phi_j\rangle\right)\phi_j
+i\frac{k\theta_j}{2\lambda_j}(\zeta-\Re\langle\zeta,\phi_j\rangle\phi_j),
\end{aligned}
\end{equation}
for Model 1, whereas for Model 2 it reads
\begin{equation}\label{eq:m2_dec}\begin{aligned}
\dot\lambda_j=&\,\frac{k}{2}\Re\langle \zeta, \phi_j \rangle (\th_j-\lambda_j^2),\\
i\partial_t\phi_j=&\,\left(H_j+\frac{k}{2}\l_j\Im\langle\zeta,\phi_j\rangle\right)\phi_j+i\frac{k\th_j}{2\l_j}\left(\zeta-\Re\langle\zeta,\phi_j\rangle\phi_j\right).
\end{aligned}\end{equation}
\begin{remark}\label{rmk:theta}
Let us notice that the dynamics for $\phi_j$ in \eqref{eq:m1_dec} can be thought of as given by sum of two contributions, a conservative and a non-conservative one. It is then clear how $\theta_j$ determines the strength of the non-conservative part, that eventually leads to synchronization. A similar remark holds for \eqref{eq:m2_dec}, too.
\end{remark}

\section{On the Cauchy Problem for Model 1 and Model 2}\label{sec:Cau}
In this section we briefly discuss the existence of local and global solutions for models \eqref{model1} and \eqref{m2}. 
Our paper mainly focuses on the asymptotic behavior of solutions, nevertheless we deem a minimal discussion on well-posedness issues is necessary. Indeed, models \eqref{model1} and \eqref{m2} bear some non-standard aspects with respect to the usual well-posedness results available in the literature, that will be highlighted below. 
For a thorough presentation of the classical results and tools used to study well-posedness problems for nonlinear Schr\"odinger-type equations we refer to \cite{Caz}.

We denote $\Psi=(\psi_1, ..., \psi_N)$, analogously $\Psi_0=(\psi_1^0,..., \psi_N^0)$ or $\Theta = (\theta_1, \dotsc, \theta_N)$.
First of all, we notice that if the $\theta_j$'s in \eqref{model1} and \eqref{m2} were given, then the local well-posedness result could be readily inferred from the self-adjointness of the Hamiltonian operators $H_j$, $j=1, ..., N$ and Stone's  theorem, see for instance \cite[Theorem VII.7]{RS}. Indeed, in that case it is possible to show the existence of solutions to models \eqref{model1} and \eqref{m2} by a standard fixed point argument in $L^2$ based spaces.

On the other hand, the evolution of the intrinsic parameters, given by the Cucker-Smale dynamics, is coupled with the Schr\"odinger equations for $\psi_j$'s via their center of masses, that are not well defined for general $L^2$ functions. Therefore it is necessary to study solutions with more regularity, namely belonging to some weighted spaces. 
For this reason we consider the functional space
\begin{equation*}
\Sigma(\R^d)=\{f\in H^1(\R^d)\,:\int|x|^2|f(x)|^2\,dx<\infty\}.
\end{equation*}
Let us recall that $\Sigma$ is the domain of the quadratic form associated to the hamornic oscillator $H=-\frac12\Delta+\frac{|x|^2}{2}$, i.e.
\begin{equation*}
\Sigma=\{f\in L^2\,:\,\langle f, Hf\rangle<\infty\}.
\end{equation*}
By using the Feynman's path integral approach \cite{Alb, Fuj, Yaj}, see also \cite{Car}, it is possible to define the propagator associated to a large class of Hamiltonian operators $H=-\frac12\Delta+V$, where $V=V(x)$ satisfies the following assumption.
\begin{assumption}\label{ass:V}
We let $V=V_\infty+V_p$, where $V_\infty$ satisfies
\begin{itemize}
\item $V_\infty\in C^\infty(\R^d)$;
\item $|\nabla^\alpha V_\infty(x)|\leq C$, for any $x\in\R^d$, $|\alpha|\geq2$,
\end{itemize}
whereas $V_p$ satisfies
\begin{equation*}
V_p, \nabla V_p\in L^p + L^\infty,\quad\;p\geq1, p>\frac{d}{2}.
\end{equation*}
\end{assumption}
Physically relevant examples covered by the above assumption are the harmonic oscillator $V(x)=\frac{\omega^2}{2}|x|^2$, or any potential $V(x)\sim|x|^{-\gamma}$ that is longer range than the Coulomb interaction $\gamma<1$, or any linear combination of such potentials.
\begin{remark}
The class of admissible potentials could be further extended, for instance by considering time-dependent potentials. Moreover, even the local well-posedness result stated below is not optimal in terms of regularity assumptions on the initial data. However, since our main focus is on the asymptotic behavior of solutions, we prefer to avoid unnecessary technicalities related to the local behavior.
\end{remark}
\begin{proposition}\label{prop:lwp}
Let $\Psi_0\in \Sigma(\R^d;\C^N)$, $\Theta_0\in\R^N_+$ with $\frac1N\sum_{j=1}^N\theta_{j, 0}=1$ and let $H_j=H+\Omega_j$, where $H=-\frac12\Delta+V$ with $V$ satisfying Assumption \ref{ass:V}, and $\Omega_j\in\R$. 
Then there exists a unique maximal solution $\Psi\in C([0, T_{max});\Sigma)$ to model \eqref{model1}, respectively, model \eqref{m2}, such that the following blow-up alternative holds true. Either $T_{max}=\infty$ or $T_{max}<\infty$ and we have
\begin{equation}\label{eq:bu_alt}
\lim_{t\to T_{max}}\sup_{j=1, ..., N}(\|\psi_j(t)\|_{\Sigma}
+\|\psi_j(t)\|_{2}^{-1})=\infty,
\end{equation}    
for Model 1, respectively
\begin{equation*}
\lim_{t\to T_{max}}\sup_{j=1, ..., N}\|\psi_j(t)\|_{\Sigma}
=\infty,
\end{equation*}    
for Model 2.
\end{proposition}
Let us emphasize that the blow-up alternative for Model 1 stated in \eqref{eq:bu_alt} is not the usual one (see \cite{Caz} for instance), due to the presence of the normalization factor in  \eqref{model1}. Indeed, the possibility for one wave function to vanish at some time would yield non-uniqueness issues in the model. A similar mathematical difficulty is considered also in \cite{ACS}, where a nonlocal renormalization term as in \eqref{model1} appears.
\newline
For this reason, in order to globally extend the solutions we also need a lower bound on the $L^2$ norm of each $\psi_j$. Let us define
\begin{equation}\label{eq:lambda_pm}
\lambda_-(t)=\min_{j=1, ..., N}\|\psi_j(t)\|_{2},\quad
\lambda_+(t)=\max_{j=1, ..., N}\|\psi_j(t)\|_{2}.
\end{equation}
It is possible to check that the following lower bound on $\lambda_-$ holds true,
\begin{align}\label{eq:lb}
    \l_-(t)\geq \l_-(0)-\frac{k}{2}\frac{\l_+(0)}{c}(\th_+-1)(0)\exp\left(\frac{k}{2c}(\th_+-1)(0)\right),
\end{align}
where $\theta_+=\max\{\theta_1, \dotsc, \theta_N\}$, see subsection \ref{QB} and in particular estimate \eqref{eq:lower_bound} for a similar calculation.
In particular the lower bound in \eqref{eq:lb} depends only on the initial data. Consequently, it is possible to prove that, as long as the right hand side of \eqref{eq:lb} is strictly positive, then the solution to model \eqref{model1} can be extended for all times.
Indeed, once a lower bound on the total masses is inferred, it is also possible to analogously derive an upper bound, see for instance estimate \eqref{eq:ub} or Lemma \ref{lem:ub} for similar calculations. By using standard tools for nonlinear Schr\"odinger-type equations (Strichartz estimates, bootstrap arguments), these in turn yield also suitable bounds on the $\Sigma$ norms of the oscillators.
\begin{proposition}\label{prop:gwp1}
Let us consider system \eqref{model1} and let the assumptions of Proposition \ref{prop:lwp} hold true. If in addition we assume that
\begin{equation}\label{eq:lb_c}
\l_-(0)-\frac{k}{2}\frac{\l_+(0)}{c}(\th_+-1)(0)\exp\left(\frac{k}{2c}(\th_+-1)(0)\right)>0,
\end{equation}
where $\lambda_-$ and $\lambda_+$ are defined in \eqref{eq:lambda_pm}, then the solutions to \eqref{model1} constructed in Proposition \ref{prop:lwp} are global in time, namely we have $\Psi\in C([0, \infty);\Sigma)$.
\end{proposition}
On the contrary, as system \eqref{m2} does not involve the renormalization factor, existence of global solutions is guaranteed by simply providing uniform bounds on the $\Sigma$ norm over compact time intervals.
\begin{proposition}\label{prop:gwp2}
Let the assumptions of Proposition \ref{prop:lwp} hold true. In addition, let $\phi_j$ be defined by \eqref{eq:ren} and let us assume that either
    \begin{equation}\label{eq:c1}
\Re\langle \phi_j,\phi_k\rangle(0)\geq 0, \ \forall \  j,k=1,,,N,
    \end{equation}
    or
    \begin{equation}\label{eq:c2}
    \|\psi_{j}^0\|_2^2\leq \th_j, \ \forall \   j=1,..,N.
    \end{equation}
Then, the solutions to model \eqref{m2} constructed in Proposition \ref{prop:lwp} are global in time, namely $\Psi\in C([0, \infty);\Sigma)$.
\end{proposition}
Even for Model 2, it is quite standard to prove that conditions \eqref{eq:c1} or \eqref{eq:c2} imply upper bounds on the $\lambda_j$'s, which in turn provide suitable bounds also for the $\Sigma$ norms.
\newline
In order to study the asymptotic behavior of the models under our consideration, we will make use of Barbalat's Lemma.
\begin{lemma}\label{BL}
    If $f(t)$ has a finite limit as $t\to \infty$, and if $\ddt f(t)$ is uniformly continuous, then $\ddt f(t) \to 0$ as $t \to \infty$.
\end{lemma}
\section{Two Oscillators}\label{sec:two}
In this section we focus on the case of two oscillators. Since for $N=2$ systems \eqref{model1} and \eqref{m2} behave in a sufficiently similar fashion, we present their qualitative properties simultaneously.
\newline
By recalling the decomposition \eqref{eq:ren} for each oscillator $\psi_j$, we infer synchronization properties of both systems by analizying the asymptotic behavior of the correlations $\Re\langle\phi_1, \phi_2\rangle$. Let us remark that by definition we have
\begin{equation*}
|\Re\langle\phi_1, \phi_2\rangle|\leq1
\end{equation*}
and
\begin{equation*}
\|\phi_1-\phi_2\|_{2}^2=2\left(1-\Re\langle\phi_1, \phi_2\rangle\right),
\end{equation*}
so that showing synchronization for $\phi_1$ and $\phi_2$ is equivalent to showing that
\begin{equation*}
\lim_{t\to\infty}\Re\langle\phi_1, \phi_2\rangle=1.
\end{equation*}
\begin{theorem}\label{synch2}
Let $N=2$ and let the assumptions of Proposition \ref{prop:gwp1} (Proposition \ref{prop:gwp2}, respectively) be satisfied. If we further assume that
\begin{equation*}
\Re\langle\phi_{1, 0}, \phi_{2, 0}\rangle\neq-1
\end{equation*}
and that the symmetric communication kernel $h(r)$ satisfies the Heavy Tail condition, then solutions to system \eqref{model1} (system \eqref{m2}, respectively) undergo complete synchronization at exponential rate, namely
\begin{equation*}
1-\Re\langle\phi_1, \phi_2\rangle\lesssim e^{-ct},\quad\textrm{for} \ \,t\to\infty.
\end{equation*}
Further, the centers of mass,  $x_1,x_2$, aggregate exponentially in time,
\begin{equation*}
    |x_1-x_2|\lesssim e^{-Ct}.
\end{equation*}
\end{theorem}
\begin{corollary}\label{bs}
Under the same assumptions of Theorem \ref{synch2}, the following bounds hold true.
For Model 1, there exist constants $C_1, C_2>0$, depending only on $(\Psi_0, \Theta_0)$, such that
\begin{equation}\label{eq:mass_bd}
C_1\leq\lambda_j(t)\leq C_2, \quad\textrm{for any}\,j=1, 2,\,t>0.
\end{equation}
For Model 2, we have
\begin{equation*}
|\lambda_j^2(t)-\theta_j(t)|\lesssim e^{-ct}, \quad\textrm{as}\,t\to\infty.
\end{equation*}
\end{corollary}
\begin{remark}
The constants $C_1, C_2$ appearing in estimate \eqref{eq:mass_bd} are explicitly determined in subsection \ref{QB}.
\end{remark}
The rest of this section is dedicated to proving Theorem \ref{synch2} and Corollary \ref{bs}.
\subsection{Synchronization}
First we will achieve synchronization of directions via a monotonicity argument. Let us compute,
\begin{align}
(\text{Model} \ 1) \ \ \ddt \Re\langle \phi_1,\phi_2\rangle&=\frac{k\th_1}{2\l_1}(\Re\langle \zeta,\phi_2\rangle-\Re\langle \zeta,\phi_1\rangle \Re\langle \phi_1,\phi_2 \rangle)+\frac{k}{2\l_1}\Im \langle \phi_1,\zeta\rangle\Im\langle\phi_1,\phi_2\rangle, \nonumber\\
&+\frac{k\th_2}{2\l_2}(\Re\langle \phi_1,\zeta\rangle-\Re\langle \zeta,\phi_2\rangle \Re\langle \phi_1,\phi_2 \rangle)+\frac{k}{2\l_2}\Im \langle \zeta,\phi_2\rangle\Im\langle\phi_1,\phi_2\rangle,\nonumber\\
(\text{Model} \ 2) \ \ \ddt \Re\langle \phi_1, \phi_2 \rangle&= \frac{k}{2} \frac{\th_1}{\lambda_1}\left(\Re\langle \zeta, \phi_2\rangle-\Re\langle \zeta,\phi_1\rangle\Re\langle \phi_1,\phi_2\rangle\right)+\frac{k}{2}\lambda_1\Im\langle \phi_1,\zeta\rangle\Im\langle \phi_1,\phi_2\rangle\nonumber,\\
&+\frac{k}{2}\frac{\th_2}{\lambda_2}\left(\Re\langle \phi_1,\zeta\rangle-\Re\langle \zeta,\phi_2\rangle\Re\langle \phi_1,\phi_2\rangle\right)+\frac{k}{2}\lambda_2\Im\langle \zeta, \phi_2\rangle\Im\langle \phi_1,\phi_2\rangle.\nonumber
\end{align}
Now, again because $N=2$, we have $\langle \phi_1,\zeta\rangle=\frac{1}{2}\l_1+\frac{1}{2}\lambda_2\langle \phi_1,\phi_2\rangle$, and a similar equality for $\langle \phi_2,\zeta\rangle$ holds, to allow us to continue,
\begin{align*}
(\text{Model} \ 1) \ \ \ddt \Re\langle \phi_1,\phi_2\rangle&=\frac{k}{4}\left(\frac{\th_1\l_2}{\l_1}+\frac{\th_2\l_1}{\l_2}\right)\left(1-(\Re\langle\phi_1,\phi_2\rangle)^2\right)+\frac{k}{4}\left(\frac{\l_2}{\l_1}+\frac{\l_1}{\l_2}\right)(\Im\langle \phi_1,\phi_2\rangle)^2,\\
(\text{Model} \ 2) \ \ \ddt \Re\langle \phi_1,\phi_2\rangle&=\frac{k}{4}\left(\frac{\th_1\lambda_2}{\lambda_1}+\frac{\th_2\lambda_1}{\lambda_2}\right)\left(1-\left(\Re\langle \phi_1,\phi_2\rangle\right)^2\right)+\frac{k}{2}\lambda_1\lambda_2\left(\Im\langle \phi_1,\phi_2\rangle\right)^2.
\end{align*}

Now we see that since $\Re\langle \phi_1,\phi_2\rangle\leq 1$, we have $\ddt \Re\langle \phi_1,\phi_2\rangle\geq 0$ and therefore $\Re\langle \phi_1,\phi_2\rangle \to c \in (-1,1]$, and hence $\ddt \Re\langle \phi_1,\phi_2\rangle \to 0$, by Barbalat's lemma. However, if $c\neq 1$, then $\ddt \Re\langle \phi_1,\phi_2\rangle>0$, thus $c=1$, as it is the only value other than $\Re\langle \phi_1,\phi_2\rangle=-1$ for which $\ddt \Re\langle \phi_1,\phi_2\rangle=0.$ Thus, for both models and any initial conditions such that $\Re\langle \phi_1,\phi_2\rangle\neq-1$, we have synchronization. \\

\subsection{Aggregation and Boundedness}
Before achieving the exponential rate of synchronization, we show aggregation of the centers of masses, as well as boundedness of the masses for Model 1. Now, from synchronization of the $\phi_j$'s, we have $\frac{\psi_j}{\|\psi_j\|_2} \to \frac{\zeta}{\|\zeta\|_2}$, which gives $x_j \to \frac{1}{\|\zeta\|_2^2}\int x|\zeta|^2 \dx$, and hence aggregation of the centers of mass ensues.

We turn our attention to the Cucker-Smale dynamics for the parameters $\th_j$,
\begin{align}
\dot{\th_j}&=\frac{\mu}{N}\sum_{k=1}^Nh(|x_j-x_k|)(\th_k-\th_j), \ \ \mu \geq 0, \nonumber\\
x_j&=\frac{1}{\|\psi_j\|_2^2}\int_{\mathbb{R}^d} x|\psi_j(x,t)|^2 \dx, \nonumber
\end{align}
since the center of masses aggregate, asymptotically $h(|x_j-x_k|) \sim h(0)$, but in particular for any influence function $h(r)$, satisfying the heavy tail condition, we have $h(|x_1-x_2|)\geq m>0$ for all time $t>0$. Thus we have \eqref{align}.

Now for Model 1, using \eqref{e:psi2} this guarantees the boundedness of each $\|\psi_j\|_2$. Indeed, let $\psi_+$ represent the larger of the two oscillators in $L^2$ norm. This can potentially change which oscillator we are referring to over time, however we can proceed using the Rademacher lemma,

\begin{align}
\ddt \l_+&=\frac{k}{2}\Re\langle \zeta,\phi_+\rangle(\th_+-1),\\
|\ddt\l_+|&\leq \frac{k}{2} \|\zeta\|_2|\th_+-1|(0)e^{-\mu mt},\\
&\leq \frac{k}{2}\l_+|\th_+-1|(0)e^{-\mu mt},\\
\l_+(t)&\leq \l_+(0)\exp\left(\frac{k}{2\mu m}|\th_+-1|(0)(1-e^{-\mu mt})\right),\label{eq:ub}\\
&\leq \l_+(0)\exp\left(\frac{k}{2\mu m}|\th_+-1|(0)\right).
\end{align}
Then for Model 2, we have \eqref{align} and the synchronization of the directions of each of the oscillators, but to see their limiting sizes, consider,
\begin{align}
\ddt\l_j=\frac{k}{2}\Re\langle \zeta, \phi_j\rangle(\th_j-\l_j^2).
\end{align}
As $\Re\langle \phi_1,\phi_2\rangle \to 1$, we have a $T\geq  0$,  such that $\Re\langle \zeta,\phi_j\rangle(t)\geq c>0$ for all $t\geq T$. Thus,
\begin{align}
\ddt (\th_j-\l_j^2)&=\frac{\mu}{2}\sum_{k=1}^2h(|x_k-x_j|)(\th_k-\th_j)-k\l_j\Re\langle \zeta,\phi_j\rangle(\th_j-\l_j^2)\\
&\leq |\th_1-\th_2|(0)e^{-\mu mt}-c(\th_j-\l_j^2),
\end{align}
and Duhamel's Principle lets us conclude $\th_j-\l_j^2\to 0$ exponentially fast, and hence $\l_j^2 \to \bar{\th}$ as well.

Note that for $\mu=0$, $\ddt \th_j=0$ and thus the intrinsic desired mass of the oscillator does not change. In this case the synchronization of directions still occurs, but if $\th_1\neq \th_2$, then $\l_j \to \sqrt{\th_j}$, still at an exponential rate, so that each oscillator has a different total mass.  The center of masses still aggregate in this case as well $|x_1-x_2|\to0$.

\subsection{Exponential Synchronization}
With bounds on the masses of each oscillator, and synchronization of the individual oscillators, we can bootstrap the rate of synchronization. We have,
\begin{align*}
(\text{Model} \ 1) \ \ &\ddt(1-\Re\langle \phi_1,\phi_2\rangle^2)=\\
&\ \ \ -\frac{k}{2}\left(\frac{\th_1\l_2}{\l_1}+\frac{\th_2\l_1}{\l_2}\right)\left(1-(\Re\langle\phi_1,\phi_2\rangle)^2\right)\Re\langle \phi_1,\phi_2\rangle-\frac{k}{2}\left(\frac{\l_2}{\l_1}+\frac{\l_1}{\l_2}\right)(\Im\langle \phi_1,\phi_2\rangle)^2\Re\langle \phi_1,\phi_2\rangle,\\
(\text{Model} \ 2) \ \ &\ddt(1-\Re\langle \phi_1,\phi_2\rangle^2)=\\
&\ \ \ -\frac{k}{2}\left(\frac{\th_1\lambda_2}{\lambda_1}+\frac{\th_2\lambda_1}{\lambda_2}\right)\left(1-\left(\Re\langle \phi_1,\phi_2\rangle\right)^2\right)\Re\langle \phi_1,\phi_2\rangle-k\lambda_1\lambda_2\left(\Im\langle \phi_1,\phi_2\rangle\right)^2\Re\langle \phi_1,\phi_2\rangle.
\end{align*}
Now since $\Re\langle \phi_1,\phi_2\rangle \to 1,$ and $\ddt \Re\langle \phi_1,\phi_2\rangle\geq 0$, there is a time, $T$,  such that $\Re\langle \phi_1,\phi_2\rangle(T)\geq c>0$ for all $t>T$. Further, with each $\l_j$ bounded and $\th_j \to 1$ (or constant in the case of Model 2 and $\mu=0$) we have,
\begin{align}
\ddt(1-\Re\langle \phi_1,\phi_2\rangle^2)&\leq -c(1-\Re\langle \phi_1,\phi_2\rangle^2),  \ \ \text{for all} \ t>T, \nonumber\\
\end{align}
and therefore synchronization occurs exponentially in time for both of the models.\\

\subsection{Quantitative Bounds}\label{QB}

Now for Model 2 we already have asymptotic dynamics for the mass of each oscillator, $\l_j^2 \sim \th_j$, but with exponential synchronization, we can establish quantitative bounds for the mass of each oscillator in Model 1. This section will provide these bounds for Model 1.

First, without loss of generality, suppose that $\l_1(0)\geq\l_2(0)$. Then by cases let us first check the case where $\th_1\geq 1$, and recall that $\th_1(t)+\th_2(t)=2.$\\

$\mathit{Case \ 1: \th_1\geq 1}$.\\

 Then we have first for the larger $\l_1$,
\begin{align}
\ddt \l_1&=\frac{k}{2}\Re\langle \zeta,\phi_1\rangle(\th_1-1)\\
&=\frac{k}{4}(\l_1+\l_2\Re\langle\phi_1,\phi_2\rangle)(\th_1-1)\geq 0
\end{align}

Further, from the Cucker-Smale dynamics we have,
\begin{align}
(\th_1-1)(t)\leq (\th_1-1)(0)e^{-\mu mt},
\end{align}
where $m=\min_th(|x_1-x_2|)$.  Thus,
\begin{align}
\ddt\l_1\leq \frac{k}{2}\l_1(\th_1-1)(0)e^{-\mu mt},
\end{align}
and therefore,
\begin{align}
\l_1(t)\leq \l_1(0)\exp\left(\frac{k}{2\mu m}(\th_1-1)(0)(1-e^{-\mu mt})\right)\leq \l_1(0)\exp\left(\frac{k}{2\mu m}(\th_1-1)(0)\right).
\end{align}

Now for $\l_2$,
\begin{align}
\ddt \l_2&=\frac{k}{2}\Re\langle \zeta,\phi_2\rangle(\th_2-1)\\
&=\frac{k}{4}(\l_1\Re\langle \phi_1,\phi_2\rangle+\l_2)(1-\th_1).
\end{align}
as $\l_2(0)\leq \l_1(0)$, it is possible for $\l_2$ also to grow if $\Re\langle \phi_1,\phi_2\rangle(0)\sim -1$. However,  if there is a time $T$ where $\l_2(T)=\l_1(T)$, we would have $\ddt \l_1(T)=-\ddt \l_2(T)$, which could only occur if $\th_1=\th_2=1$, (which would simply be the original Schrodinger-Lohe model) or if $\Re\langle \phi_1,\phi_2\rangle(T)=-1$,  which can only occur if $\ddt \Re\langle \phi_1,\phi_2\rangle(0)=-1$, as $\ddt \Re\langle \phi_1,\phi_2\rangle \geq 0.$ Therefore $\l_2(t)\leq \l_1(t)\leq \l_1(0)\exp\left(\frac{k}{2\mu m}(\th_1-1)(0)\right).$

Then to achieve a lower bound on $\l_2$, we note that $\Re\langle \zeta, \phi_2\rangle \leq \l_1$, then, using the bound we already achieved for $\l_1$ and $1-\th_1$,
\begin{align}\label{eq:lower_bound}
\ddt \l_2 &\geq \frac{k}{2}\l_1(1-\th_1)\\
&\geq \frac{k}{2}\l_1(0)\exp\left(\frac{k}{2\mu m}(\th_1-1)(0)\right)(1-\th_1)(0)e^{-\mu mt}.
\end{align}
Then integrating we get,
\begin{align}
\l_2(t)&\geq \l_2(0)-\frac{k}{2}\frac{\l_1(0)}{\mu m}(\th_1-1)(0)\exp\left(\frac{k}{2\mu m}(\th_1-1)(0)\right)\left(1-e^{-\mu mt}\right)\\
&\geq  \l_2(0)-\frac{k}{2}\frac{\l_1(0)}{\mu m}(\th_1-1)(0)\exp\left(\frac{k}{2\mu m}(\th_1-1)(0)\right).
\end{align}

$\mathit{Case \ 2: \th_1< 1}$.\\

Now, if $\th_1<1$, we note that since $\l_1(0)>\l_2(0)$, $\Re\langle \zeta,\phi_1\rangle>(0)>0$, and therefore $\ddt \l_1 \leq 0$. If there exists a time $T$ such that $\l_1(T)=\l_2(T)$, then at this time, $\ddt \l_1(T)=-\ddt \l_2(T)$, and hence for $t\geq T, \l_2(t)\geq \l_1(t)$, and we are in the first case with the roles of $\l_1$ and $\l_2$ reversed. In the case that $\l_1(t) > \l_2(t)$ for all $t$, we proceed as before to get,
\begin{align}
\ddt \l_1 &\geq \frac{k}{2} \l_1(\th_1-1)(0)e^{-\mu mt}\\
\l_1(0)\geq \l_1(t)&\geq \l_1(0)\exp\left(\frac{k}{2\mu m}(\th_1-1)(0)\right),
\end{align}
and for $\l_2$,
\begin{align}
\l_2(t)\leq \l_2(0)+\frac{k}{2}\frac{\l_1(0)}{\mu m}(1-\th_1)(0)\exp\left(\frac{k}{2 \mu m}(\th_1-1)(0)\right).
\end{align}
This is true for all initial conditions except for if $\phi_1(0)=\pm\phi_2(0)$. In which case $\ddt \Re\langle \phi_1,\phi_2\rangle=0$ for all time. Hence, assuming without loss of generality that $\l_1(0)>\l_2(0)$, we have $\phi_1=\frac{\zeta}{\|\zeta\|}=\pm\phi_2$. Then the only thing to observe is the mass of each oscillator, using $\frac{1}{2}(\th_1+\th_2)=1$,
\begin{align}
\ddt \l_1&=\frac{k}{2}\|\zeta\|(\th_1-1),\\
\ddt \l_2&=\pm\frac{k}{2}\|\zeta\|(\th_2-1)=\mp\ddt \l_1,
\end{align}
where, $\ddt \l_1=-\ddt \l_2$ if $\phi_1=\phi_2$ and $\ddt \l_1=\ddt \l_2$ if $\phi_1=-\phi_2$.
Hence, in either case, $\ddt \|\zeta\|=0$, and since $x_1=x_2$, we have exponential alignment of $\th_j$ so that $\ddt\|\psi_j\| \to 0$ exponentially fast, ensuring boundedness of each $\|\psi_j\|$.

In particular, the same bounds above hold with $m=h(0)$.\\

The final note about the case of $N=2$ oscillators with identical natural frequencies, is the existence of the unstable equilibrium at $\phi_1=-\phi_2$. Which from $\ddt \Re \langle \phi_1,\phi_2\rangle>0$ for $\phi_1 \neq \pm \phi_2$, is of course not stable, and further can only occur if initially so, $\phi_1(0)=-\phi_2(0)$.  This concludes the proof of Theorem \ref{synch2} and its Corollary \ref{bs}.

\subsection{Different Natural frequencies}

To conclude the analysis of the two oscillator case, we can investigate the synchronization dynamics when the two oscillators have differing natural frequencies, as well. The following analysis only applies to Model 1. Consider,
\begin{align}\label{dfreq}
i\partial_t \psi_1&=H\psi_1+\O\psi_1+\frac{ik}{4}\left(\th_1(\psi_1+\psi_2)-(\psi_1+\frac{\langle \psi_2,\psi_1\rangle}{\|\psi_1\|_2^2} \psi_1)\right),\\
i\partial_t \psi_2&=H\psi_2-\O\psi_2+\frac{ik}{4}\left(\th_2(\psi_1+\psi_2)-(\frac{\langle \psi_1,\psi_2\rangle}{\|\psi_2\|_2^2} \psi_2+\psi_2)\right)\nonumber.
\end{align}
along with $\th_j$ controlled by the Cucker-Smale dynamics with absolute kernel, $\inf h(r)=c>0.$
Indeed we see that synchronization will still occur depending on the relative size of the natural frequency $\O$ and the masses of the oscillators, however the limiting state will be different.

First, notice that the absolute kernel immediately yields \eqref{align} and thus the quantitative bounds established in the previous section hold. Again splitting the oscillator into $\psi_j(x,t)=\l_j(t)\phi_j(x,t)$, where $\l_j(t)=\|\psi_j\|_2(t)$,  we let $\bar{\l}_j=\lim_t \l_j(t)$ for each oscillator and let $z(t)=\langle \phi_1,\phi_2\rangle$ be the correlation between the two oscillator directions. Then the synchronization dynamics can be characterized by the size of $\L=\frac{4\O\bar{\l}_1\bar{\l}_2}{k(\bar{\l}_1^2+\bar{\l}_2^2)}$.
\begin{theorem}\label{thm:nio}
Let $\psi_{1,0},\psi_{2,0} \in L^2(\mathbb{R}^d).$ Then depending on the relative size of the coupling constant, the limiting masses, and the natural frequency, the dynamics are as follows

(1) for $0\leq \L<1$ either complete synchronization ensues at an exponential rate,
\begin{align}
    1-|z(t)| \lesssim e^{-Ct}.
\end{align}
as $t \to \infty$ with limiting correlation
\begin{align}
    |z(t)-z_1|\lesssim e^{-Ct},
\end{align}
$z_1=\sqrt{1-\L^2}+i\L$, or $z(t)$ converges to an unstable equilibrium $z_2=-\sqrt{1-\L^2}+i\L$\\

(2) for $\L=1$, complete synchronization occurs at a linear rate,
\begin{align}
    1-|z(t)|\lesssim t^{-1},
\end{align}
as $t \to \infty$, with limiting correlation $z_1=z_2=i$,\\

(3) for $\L>1$,

the correlations tend towards periodic behavior.
\end{theorem}

First, the boundedness and convergence of $\l_1,\l_2$ are not changed by the inclusion of differing natural frequencies. Therefore none of the three possible configurations of parameters are empty.

In order to prove the theorem we will appeal to the fact that we can explicitly solve the system up to an exponentially decaying term, analogous to what is shown in \cite{AM}. First for $\L<1$ and for $z(t)=\langle \phi_1,\phi_2\rangle$ we compute,
\begin{align}
\ddt z(t)=2i\O z+\frac{k}{4}\left(\frac{\th_1\l_2^2+\th_2\l_1^2}{\l_1\l_2}\right)(1-z^2)+\frac{ik}{4}\left(\frac{\l_2}{\l_1}(\th_1-1)+\frac{\l_1}{\l_2}(\th_2-1)\right)\Im(z)z
\end{align}
Now as $\th_1,\th_2$ are controlled by the Cucker-Smale dynamics under an absolute kernel, we not only have $\th_1,\th_2 \to 1$ exponentially fast, but also $\l_1,\l_2 \to \bar{\l}_1,\bar{\l}_2$, all at an exponential rate, and hence,
\begin{align}
\ddt z(t)=2i\O z+\frac{\O}{\L}(1-z^2)+E_1(z,t)
\end{align}
for $E_1(z,t)$ an exponentially decaying quantity that depends on $z$, but due to $0\leq |z|\leq 1$, it is bounded above and below by exponentially decaying quantities that do not depend on $z$, $E_0(t)\leq |E_1(z,t)| \leq E_2(t)$. Now this is a perturbation of the differential equation,
\begin{align}
\ddt y(t)=2i\O y+\frac{\O}{\L}(1-y^2),
\end{align}
which, can be shown to have solution,
\begin{align}
y(t)=\frac{z_1+\bar{z}_1\frac{y(0)-z_1}{y(0)+\bar{z}_1}e^{\frac{-2\O}{\L}\sqrt{1-\L^2}t}}{1-\frac{y(0)-z_1}{y(0)+z_1}e^{\frac{-2\O}{\L}\sqrt{1-\L^2}t}}
\end{align}
so long as $y(0)\neq z_2$.

Then we compute
\begin{align}
    \ddt (z(t)-y(t))&=2i\O(z-y)+\frac{\O}{\L}(y-z)(y+z)+E_1(z,t),\\
    &=\O(z-y)\left(2i-\frac{1}{\L}(y+z)\right)+E_1(z,t).
\end{align}
Indeed for large times, where $|y(T)-z_1|<\e_1$ and $|E_1(z,T)|<\e_2$ we see that the difference $z(t)-y(t)$ behaves monotonically away from the asymptotic fixed points, $z_1$ and $z_2$ where the monotonicity can break down near these points due to the small fluctuations of $E_1(z,t)$. Indeed for $-\sqrt{1-\L^2}<<\Re{z(T)}$ Duhamel's principle gives,
\begin{align}
    z(t)-y(t)&=(z(T)-y(T))\exp{\left(i\O\int_T^t 2-\frac{1}{\L}(\Im{y(s)}+\Im{z(s)})\ds\right)}\exp{\left(-\O\int_T^t\Re{y(s)}+\Re{z(s)} \ds\right)}\\
    &+\int_T^tE_1(z,s)\exp{\left(\O\int_s^t 2i-\frac{1}{\L}(y(\t)+z(\t))\dtau\right)}
\end{align}
and clearly $\Re{y(s)}+\Re{z(s)}>0$ and $z(t)-y(t)$ converges to zero exponentially in time.

Now if $\Re{z(T)}<< -\sqrt{1-\L^2}$ then $z(t)-y(t)$ actually grows as $\Re{y(s)}+\Re{z(s)}<0$, which means $z(t)$ is approaching the unstable asymptotic fixed point, $z_2$. Now this fixed point is unstable as the perturbation from $E_1(z,t)$ could push $z$ so that it begins to approach $z_1$.

Therefore, for $\L<1$, $z(t)$ either converges to the unstable asymptotic fixed point $z_2$, or converges exponentially fast to the stable fixed point $z_1$.\\

Now let $\L=1$. Then,
\begin{align}
    \ddt z(t)=\O(1+2iz-z^2)+E_1(z,t)
\end{align}
and the corresponding unperturbed problem,
\begin{align}
    \ddt y(t)=\O(1+2iy-y^2),
\end{align}
which has solution
\begin{align}
    y(t)=i+\left(\O t+\frac{1}{y(0)-i}\right)^{-1}
\end{align}

Then proceeding as in the previous case,
\begin{align}
    \ddt z(t)-y(t)=\O(z-y)(2i-(y+z))+E_1(z,t)
\end{align}

However, there is only one possible asymptotic fixed point at $z_1=z_2=i$, therefore there is no point where $z(t)$ can get `stuck'. Therefore, $z(t) \to i$ as well, but at the same rate $t^{-1}$ as $y(t)$ converges to the fixed point. Lastly, for $\L>1,$ there are no potential fixed points and $z(t)$ will tend towards periodic behavior.

This concludes the analysis of the models with two oscillators. For the remainder of the paper, $N\geq 2$, and all natural frequencies will be identical, $\O_j=0$.

\section{Model 1: A perturbation of the Schr\"odinger-Lohe Model}\label{sec:m1}
To consider the model with more than two oscillators, $N\geq 2$, we will investigate each model separately, as with more oscillators involved the dynamics and techniques needed are different. Proposition \ref{prop:gwp1} grants global well-posedness of Model 1 for any initial configuration.
\subsection{Synchronization with the Absolute Kernel} We will first consider the Model 1 system, with an absolute kernel for the Cucker-Smale influence function, i.e., $h$ satisfies, $\inf h(r)= c>0$. In this case, we can prove synchronization, for any initial conditions, to either full synchronization, or an unstable, Bipolar state. Let us rewrite the system here:
\begin{align}\label{1abs}
i\partial_t \psi_j&=H\psi_j+\frac{ik}{2N}\sum_{l=1}^N\left(\th_j\psi_l-\frac{\langle \psi_l,\psi_j\rangle}{\|\psi_j\|_2^2} \psi_j\right),\nonumber\\
\psi_j(x,0)&=\psi_j^0(x), \ \ j=1,\dots,N,\nonumber\\
\dot{\th_j}&=\frac{\mu}{N}\sum_{k=1}^Nh(|x_j-x_k|)(\th_k-\th_j), \ \ \mu > 0,\\
\inf_{r\geq 0}h(r)&=c>0,\nonumber\\
x_j&=\frac{1}{\|\psi_j\|_2^2}\int_{\mathbb{R}^d} x|\psi_j(x,t)|^2 \dx, \nonumber
\end{align}

\begin{theorem}\label{t:synch1}
Solutions to the system \eqref{1abs} undergo either complete synchronization at an exponential rate, 
\begin{align}
    1-\Re\langle \phi_+,\phi_-\rangle(t)\lesssim e^{-Ct},
\end{align}
or converge to an unstable Bipolar state,
\begin{align}
    \lim_{t\to\infty}\Re\langle \phi_j,\phi_k\rangle=\pm 1, \ \ \forall j,k=1,...,N.
\end{align}
Further, the absolute kernel in the Cucker-Smale dynamics provides exponential aggregation of centers of masses for both the stable and unstable states.
\end{theorem}

Immediately, from the classical Cucker-Smale theory, since $\inf h(r)=c>0$ we have alignment of the $\th_j \to \bar{\th}=1$, exponentially fast, regardless of location of each $x_j$, so that once again we have \eqref{align}.

Further, by observing the $L^2$ norm of the order parameter, $\zeta$, differentiating
\begin{equation}\label{e:mon}
\ddt \|\zeta\|_2^2=k\left(\bar{\th}\|\zeta\|_2^2-\frac{1}{N}\sum_{j=1}^N \Re[\langle \phi_j, \zeta\rangle^2]\right),
\end{equation}
and, since $\bar{\th}=1$, and from the Cauchy-Schwarz inequality, and the fact that $\|\phi_j\|_2=1$, we have $\Re[\langle \phi_j, \zeta\rangle^2]\leq \|\zeta\|_2^2$, which implies $\ddt \|\zeta\|_2^2 \geq 0$. Since $\|\zeta\|_2^2$ is monotonic in time, $\lim_{t \to \infty}\|\zeta\|_2^2 =\rho_{\infty}$, for $\rho_{\infty} \in(0,\infty]$,  similar to the result of Ha et al in \cite{HHK}. However, $\|\zeta\|_2$ is not a priori bounded, thus we must show that $\rho_{\infty}$ is bounded away from infinity.

To see that $\rho_\infty<\infty$ we must establish a bound on each $\l_j=\|\psi_j\|_2$.

\begin{lemma}\label{lem:ub}
The mass of each oscillator remains bounded in time, i.e., there exists a $C>0$ such that $\l_j\leq C$ for all $j=1,...,N$ and $t>0$.
\end{lemma}
\begin{proof}
Let $\psi_+(x,t)$ represent the oscillator for which $\max_{j}\l_j(t)$ is achieved. The index $+$ for $\th_+$ and $\phi_+$ correspond to the oscillator $\psi_+$ determined by $\max_{j}\l_j(t)$. Note that the oscillator $\psi_+(x,t)$ pertains to may change, thus we must employ Rademacher's lemma in order to compute
\begin{align}
\ddt \l_+&=\frac{k}{2}\Re\langle \zeta, \phi_+\rangle(\th_+(t)-1),\\
|\ddt \l_+|&\leq \frac{k}{2}\|\zeta\|_2|\th_+(t)-1|, \nonumber
\end{align}
but since $\zeta=\frac{1}{N}\sum_{j=1}^N\psi_j$, we have $\|\zeta\|_2\leq \frac{1}{N}\sum_{j=1}^N\l_j\leq \frac{1}{N}\sum_{j=1}^N\l_+=\l_+$. Therefore,
\begin{align}
|\ddt \l_+|\leq \frac{k}{2}\l_+|\th_+(0)-1|e^{-c\mu t},
\end{align}
using \eqref{align}, and
\begin{align}
\l_+(t)\leq \l_+(0)\exp\left(\frac{k}{2c\mu}|\th_+-1|(0)\right)=C
\end{align}
and therefore the mass of every oscillator remains bounded for all time.
\end{proof}
With this in hand we are ready to prove the theorem.
\begin{proof}
With the $L^2$ norm of each oscillator bounded, the $L^2$ norm of the order parameter must also be bounded and thus $\rho_{\infty}<\infty$.

Therefore we know $\ddt \|\zeta\|_2^2 \to 0$, by  Barbalat's Lemma, which of course can only occur if $\phi_j \to \pm\frac{\zeta}{\|\zeta\|_2}$ for each $j$. Therefore, we know that either Bipolar, or Full synchronization must occur for any initial condition.

Now with $\phi_j \to \pm\frac{\zeta}{\|\zeta\|}$ it is easy to see that $x_j \to \frac{1}{\rho_{\infty}}\int_{\mathbb{R}^d}x|\zeta|^2 \dx$ for each $j$, hence aggregation of the centers of masses.\\

Now to see that in the case of Full synchronization, that it must occur exponentially in time, we note that for all $j,k$, we have $\Re\langle \phi_j,\phi_k\rangle \to 1$, and in particular, there exists a time $T$ such that for all $t\geq T$, the pair $\phi_+,\phi_-$ that minimizes, $\Re\langle \phi_j,\phi_k\rangle$, satisfies,
\begin{align}
\Re\langle \phi_+,\phi_-\rangle(t)\geq C>0, \ \ \text{for \ all} \ t \geq T.
\end{align}
Therefore, just as in the $N=2$ case, we have,
\begin{align}
\ddt \Re\langle \phi_+,\phi_-\rangle&=\frac{k\th_+}{2\l_+}(\Re\langle \zeta,\phi_-\rangle-\Re\langle \zeta,\phi_+\rangle \Re\langle \phi_+,\phi_-\rangle)+\frac{k}{2\l_+}\Im \langle \phi_+,\zeta\rangle\Im\langle\phi_+,\phi_-\rangle \nonumber\\
&+\frac{k\th_-}{2\l_-}(\Re\langle \zeta,\phi_+\rangle-\Re\langle \zeta,\phi_-\rangle \Re\langle \phi_+,\phi_- \rangle)+\frac{k}{2\l_-}\Im \langle \zeta,\phi_-\rangle\Im\langle\phi_+,\phi_-\rangle.\nonumber
\end{align}

Now since $\Re\langle \phi_+,\phi_-\rangle\geq C>0$, we also have, $\Re\langle \zeta,\phi_{\pm}\rangle\geq C\|\zeta\|_2>0$ and further,
\begin{equation}\label{e:wedgeineq}
\|\zeta\|_2\Re\langle \phi_+,\phi_-\rangle\leq \Re\langle \zeta,\phi_{\pm}\rangle\leq \|\zeta\|_2.
\end{equation}
Now, since $\phi_{\pm}$ minimizes the real part of the correlation between two oscillator directions, it follows that $\phi_+$ and $\phi_-$ lie on opposite sides with respect to the order parameter $\zeta$ on the time frame $t \geq T$. Without loss of generality, assume $\Im\langle \phi_+,\zeta\rangle\geq 0$, then $\Im\langle \zeta,\phi_-\rangle\geq0$ and $\Im\langle \phi_+,\phi_-\rangle\geq 0$. To see that synchronization in fact occurs at an exponential rate, we add and subtract appropriately, and use \eqref{e:wedgeineq},
\begin{align}
\ddt \Re\langle \phi_+,\phi_-\rangle&=\frac{k\th_+}{2\l_+}(\Re\langle \zeta,\phi_-\rangle-\Re\langle \zeta,\phi_+\rangle \Re\langle \phi_+,\phi_-\rangle)+\frac{k}{2\l_+}\Im \langle \phi_+,\zeta\rangle\Im\langle\phi_+,\phi_-\rangle \nonumber\\
&+\frac{k\th_-}{2\l_-}(\Re\langle \zeta,\phi_+\rangle-\Re\langle \zeta,\phi_-\rangle \Re\langle \phi_+,\phi_- \rangle)+\frac{k}{2\l_-}\Im \langle \zeta,\phi_-\rangle\Im\langle\phi_+,\phi_-\rangle,\nonumber\\
&\geq \frac{k}{2}(1-\Re\langle \phi_+,\phi_-\rangle)(\frac{\th_+}{\l_+}\Re\langle \zeta,\phi_-\rangle+\frac{\th_-}{\l_-}\Re\langle \zeta, \phi_+\rangle)\\
&+\frac{k}{2}\Re\langle \phi_+,\phi_-\rangle(\Re\langle \zeta,\phi_-\rangle-\Re\langle\zeta,\phi_+\rangle)(\frac{\th_+}{\l_+}-\frac{\th_-}{\l_-})\nonumber\\
\end{align}
If $\frac{\th_+}{\l_+}\geq \frac{\th_-}{\l_-}$, then,
\begin{align}
\ddt \Re\langle \phi_+,\phi_-\rangle&\geq \frac{k}{2}\|\zeta\|_2(1-\Re\langle \phi_+,\phi_-\rangle)(\frac{\th_-}{\l_-}(\Re\langle \phi_+,\phi_-\rangle+\Re\langle\frac{\zeta}{\|\zeta\|_2},\phi_+\rangle) +\frac{\th_+}{\l_+}(\Re\langle \frac{\zeta}{\|\zeta\|_2}, \phi_-\rangle-\Re\langle \phi_+,\phi_-\rangle))\nonumber\\
&\geq\frac{k}{2}\|\zeta\|_2(1-\Re\langle \phi_+,\phi_-\rangle)(\frac{\th_-}{\l_-}(\Re\langle \phi_+,\phi_-\rangle+\Re\langle\frac{\zeta}{\|\zeta\|_2},\phi_+\rangle))
\end{align}
and if $\frac{\th_+}{\l_+}<\frac{\th_-}{\l_-}$, then
\begin{align}
\ddt \Re\langle \phi_+,\phi_-\rangle&\geq \frac{k}{2}\|\zeta\|_2(1-\Re\langle \phi_+,\phi_-\rangle)(\frac{\th_+}{\l_+}(\Re\langle \phi_+,\phi_-\rangle+\Re\langle\frac{\zeta}{\|\zeta\|_2},\phi_-\rangle) +\frac{\th_-}{\l_-}(\Re\langle \frac{\zeta}{\|\zeta\|_2}, \phi_+\rangle-\Re\langle \phi_+,\phi_-\rangle))\nonumber\\
&\geq\frac{k}{2}\|\zeta\|_2(1-\Re\langle \phi_+,\phi_-\rangle)(\frac{\th_+}{\l_+}(\Re\langle \phi_+,\phi_-\rangle+\Re\langle\frac{\zeta}{\|\zeta\|_2},\phi_-\rangle))
\end{align}
In either case we have,
\begin{align}
\ddt (&1-\Re\langle \phi_+,\phi_-\rangle)\leq -C(1-\Re\langle \phi_+,\phi_-\rangle),\nonumber\\
(&1-\Re\langle \phi_+,\phi_-\rangle)(t)\leq (1-\Re\langle \phi_+,\phi_-\rangle(T))e^{-Ct}.
\end{align}
for all $t\geq T.$ Thus concludes the proof.
\end{proof}
\subsection{Classification of Steady States}

Now let us classify the steady states of the system. From $\ddt \|\zeta\|_2\geq 0$, and $\ddt \|\zeta\|_2=0$ only when each $\phi_j=\pm\frac{\zeta}{\|\zeta\|_2}$, we see that all possible combinations of $\phi_j=\pm\frac{\zeta}{\|\zeta\|_2}$ correspond directly to all of the steady states of the system, other than the repulsive incoherent state $\zeta=0$. In particular let us show that $\phi_j=+\frac{\zeta}{\|\zeta\|_2}$ for all $j$ is the only stable equilibrium of the system.\\

First note that as each $\th_j \to 1$ exponentially fast, and each $\l_j$ is bounded, and thus $\|\zeta\|_2$ as well, along with $\phi_j \to \pm\frac{\zeta}{\|\zeta\|_2}$ we have for all $j$, $\l_j(t) \to \bar{\l}_j.$ Thus full synchronization is represented by,
\begin{align}
\zeta_F=\frac{1}{N}\left(\sum_{k=1}^N\bar{\l}_k\right)\phi_{\zeta}
\end{align}
with $\|\phi_{\zeta}\|=1$, and $\phi_{\zeta}$ denoting the direction of the order parameter. Then any Bipolar synchronized state can be represented,
\begin{align}
\zeta_B=\frac{1}{N}\left(\sum_{k \in I} \bar{\l}_k-\sum_{k \not\in I} \bar{\l}_k\right)\phi_{\zeta}
\end{align}
with $I=\{k : \phi_k \to +\frac{\zeta}{\|\zeta\|}=\phi_{\zeta}\}$, note that $\sum_{k \in I} \bar{\l}_k-\sum_{k \not\in I} \bar{\l}_k>0$ as $\ddt \|\zeta\|_2\geq 0$. Now to see that each Bipolar state is an unstable equilibrium, we suppose that one oscillator is perturbed off of the antipole. Then this position where $\phi_1\neq \phi_{\zeta}$ represents the direction of the oscillator that is not synchronized, and $\phi$ will represent the direction of the other $N-1$ oscillators, as with the perturbation it will no longer be the same as $\phi_{\zeta}$.
\begin{align}
\zeta_{B^{\prime}}=\frac{1}{N}\left(\sum_{k \in I} \bar{\l}_k-\sum_{k \not\in I,k\neq 1} \bar{\l}_k\right)\phi -\frac{1}{N}\bar{\l}_1\phi_1.
\end{align}
Then comparing the $L^2$ norm of both $\zeta_B$ and $\zeta_{B^{\prime}}$,
\begin{align}
\|\zeta_B\|_2&=\frac{1}{N}\left(\sum_{k \in I} \bar{\l}_k-\sum_{k \not\in I} \bar{\l}_k\right),\\
\|\zeta_{B^{\prime}}\|_2&=\left\|\left(\sum_{k \in I} \bar{\l}_k-\sum_{k \not\in I,k\neq 1} \bar{\l}_k\right)\phi -\frac{1}{N}\bar{\l}_1\phi_1\right\|_2,\\
&\geq \left|\left\|\frac{1}{N}\sum_{k \in I} \bar{\l}_k-\sum_{k \not\in I,k\neq 1} \bar{\l}_k\right\|_2-\|\frac{1}{N}\bar{\l}_1\phi_1\|_2\right|\\
&=\|\zeta_B\|_2.
\end{align}
where the inequality comes from the reverse triangle inequality and equality can only hold if $\phi_1=\phi$. Therefore, $\|\zeta_{B^{\prime}}\|_2>\|\zeta_B\|_2$, with $\ddt \|\zeta\|_2>0$ for any nonsynchronized state. Implying that each Bipolar equilibrium must indeed be unstable.

\subsection{Synchronization with the heavy tail kernel}
A key component in achieving synchronization in the previous section was knowing a priori that each $\th_j \to 1$ exponentially fast. This was derived from the usage of the absolute kernel in the Cucker-Smale dynamics appended to the system. If we loosen this condition to the classical Heavy Tail condition so that rather than $\inf_r h(r)=c>0$, we have $h(r)> 0$ and
\begin{equation}\label{HT}
\int^{\infty} h(r) \dr=\infty,
\end{equation}
then we can still achieve synchronization, but now it depends on the initial configuration of the corresponding oscillators.
\begin{theorem}\label{thm:synch3}
The system \eqref{model1} with kernel satisfying \eqref{HT} undergoes complete synchronization for initial conditions such that all oscillators satisfy the following wedge condition
\begin{equation}\label{e:wedge}
\Re\langle \phi_j,\phi_k\rangle(0)\geq 0, \ \ \text{for all} \ j,k.
\end{equation}
Further, each $\th_j \to \bar{\th}=1$, exponentially fast, and each center of mass aggregates $x_j \to \frac{1}{\rho_\infty}\int_{\R^d}x|\zeta|^2\dx$, analogous to system \eqref{1abs}. Further the rate of synchronization is exponential,
\begin{align}
    1-\Re\langle \phi_j,\phi_k\rangle\lesssim e^{-Ct}, \ \ \forall j,k.
\end{align}
\end{theorem}
\begin{proof}
First we note that we still have \eqref{e:mon} and therefore $\|\zeta\|_2\to \rho_{\infty}$. However, with the heavy tail kernel we do not immediately have $\th_j\to1$ exponentially in time, and thus do not have an a priori bound on the mass of each oscillator. Therefore, we first must show that $|x_j-x_k|(t)$ remains bounded for all time (in fact we show $|x_j-x_k|\to 0$).\\
 
Now just as in the $N=2$ case we observe,
\begin{align}
\ddt \Re\langle \phi_+,\phi_-\rangle&=\frac{k\th_+}{2\l_+}(\Re\langle \zeta,\phi_-\rangle-\Re\langle \zeta,\phi_+\rangle \Re\langle \phi_+,\phi_-\rangle)+\frac{k}{2\l_+}\Im \langle \phi_+,\zeta\rangle\Im\langle\phi_+,\phi_-\rangle \nonumber\\
&+\frac{k\th_-}{2\l_-}(\Re\langle \zeta,\phi_+\rangle-\Re\langle \zeta,\phi_-\rangle \Re\langle \phi_+,\phi_- \rangle)+\frac{k}{2\l_-}\Im \langle \zeta,\phi_-\rangle\Im\langle\phi_+,\phi_-\rangle\nonumber
\end{align}
where $\phi_{\pm}$ is where $\min_{j,k} \Re\langle \phi_j,\phi_k\rangle$ is achieved.

Now since $\Re\langle \phi_+,\phi_-\rangle\geq 0$, we also have, $\Re\langle \zeta,\phi_{\pm}\rangle\geq c>0$ and further, we have \eqref{e:wedgeineq}.

Now, since $\phi_{\pm}$ minimizes the real part of the correlation between two oscillator directions, it follows that $\phi_+$ and $\phi_-$ lie on opposite sides with respect to the order parameter $\zeta$. Without loss of generality, assume $\Im\langle \phi_+,\zeta\rangle\geq 0$, then $\Im\langle \zeta,\phi_-\rangle\geq0$ and $\Im\langle \phi_+,\phi_-\rangle\geq 0$. Therefore, $\ddt \Re \langle\phi_+,\phi_-\rangle \geq 0$, which implies that not only is \eqref{e:wedge} preserved in time, but that coupled with $\Re\langle \phi_+,\phi_-\rangle\leq 1$, and $\ddt \Re\langle \phi_+,\phi_-\rangle > 0$ if $\Re\langle \phi_+,\phi_-\rangle \neq \pm 1$, we have $\Re\langle \phi_+,\phi_-\rangle \to 1$.  Thus all the directions synchronize so that $\phi_j=\frac{\psi_j}{\|\psi_j\|_2}\to\frac{\zeta}{\|\zeta\|_2}$ for all $j$. To see that this in fact occurs at an exponential rate, we add and subtract appropriately, and use \eqref{e:wedgeineq},
\begin{align}
\ddt \Re\langle \phi_+,\phi_-\rangle&=\frac{k\th_+}{2\l_+}(\Re\langle \zeta,\phi_-\rangle-\Re\langle \zeta,\phi_+\rangle \Re\langle \phi_+,\phi_-\rangle)+\frac{k}{2\l_+}\Im \langle \phi_+,\zeta\rangle\Im\langle\phi_+,\phi_-\rangle \nonumber\\
&+\frac{k\th_-}{2\l_-}(\Re\langle \zeta,\phi_+\rangle-\Re\langle \zeta,\phi_-\rangle \Re\langle \phi_+,\phi_- \rangle)+\frac{k}{2\l_-}\Im \langle \zeta,\phi_-\rangle\Im\langle\phi_+,\phi_-\rangle\nonumber\\
&\geq \frac{k}{2}(1-\Re\langle \phi_+,\phi_-\rangle)(\frac{\th_+}{\l_+}\Re\langle \zeta,\phi_-\rangle+\frac{\th_-}{\l_-}\Re\langle \zeta, \phi_+\rangle)\\
&+\frac{k}{2}\Re\langle \phi_+,\phi_-\rangle(\Re\langle \zeta,\phi_-\rangle-\Re\langle\zeta,\phi_+\rangle)(\frac{\th_+}{\l_+}-\frac{\th_-}{\l_-})\nonumber\\
\end{align}
If $\frac{\th_+}{\l_+}\geq \frac{\th_-}{\l_-}$, then,
\begin{align}
\ddt \Re\langle \phi_+,\phi_-\rangle&\geq \frac{k}{2}\|\zeta\|_2(1-\Re\langle \phi_+,\phi_-\rangle)(\frac{\th_-}{\l_-}(\Re\langle \phi_+,\phi_-\rangle+\Re\langle\frac{\zeta}{\|\zeta\|_2},\phi_+\rangle) +\frac{\th_+}{\l_+}(\Re\langle \frac{\zeta}{\|\zeta\|_2}, \phi_-\rangle-\Re\langle \phi_+,\phi_-\rangle))\nonumber\\
&\geq\frac{k}{2}\|\zeta\|_2(1-\Re\langle \phi_+,\phi_-\rangle)(\frac{\th_-}{\l_-}(\Re\langle \phi_+,\phi_-\rangle+\Re\langle\frac{\zeta}{\|\zeta\|_2},\phi_+\rangle))
\end{align}
and if $\frac{\th_+}{\l_+}<\frac{\th_-}{\l_-}$, then
\begin{align}
\ddt \Re\langle \phi_+,\phi_-\rangle&\geq \frac{k}{2}\|\zeta\|_2(1-\Re\langle \phi_+,\phi_-\rangle)(\frac{\th_+}{\l_+}(\Re\langle \phi_+,\phi_-\rangle+\Re\langle\frac{\zeta}{\|\zeta\|_2},\phi_-\rangle) +\frac{\th_-}{\l_-}(\Re\langle \frac{\zeta}{\|\zeta\|_2}, \phi_+\rangle-\Re\langle \phi_+,\phi_-\rangle))\nonumber\\
&\geq\frac{k}{2}\|\zeta\|_2(1-\Re\langle \phi_+,\phi_-\rangle)(\frac{\th_+}{\l_+}(\Re\langle \phi_+,\phi_-\rangle+\Re\langle\frac{\zeta}{\|\zeta\|_2},\phi_-\rangle))
\end{align}
In either case we have,
\begin{align}
\ddt (&1-\Re\langle \phi_+,\phi_-\rangle)\leq -c(1-\Re\langle \phi_+,\phi_-\rangle),\nonumber\\
(&1-\Re\langle \phi_+,\phi_-\rangle)(t)\leq (1-\Re\langle \phi_+,\phi_-\rangle(0))e^{-ct}.
\end{align}

With exponential synchronization comes aggregation of the centers of masses, $x_j \to \frac{1}{\rho_\infty}\int_{\R^d}x|\zeta|^2\dx$, exponentially fast as well.

Returning to the Cucker-Smale dynamics,
\begin{align}
\dot{\th_j}&=\frac{\mu}{N}\sum_{k=1}^Nh(|x_j-x_k|)(\th_k-\th_j), \ \ \mu > 0, \nonumber\\
x_j&=\frac{1}{\|\psi_j\|_2^2}\int_{\mathbb{R}^d} x|\psi_j(x,t)|^2 \dx, \nonumber
\end{align}
since the center of masses aggregate, asymptotically $h(|x_j-x_k|) \sim h(0)$, but in particular $h(|x_j-x_k|)\geq m>0$ for all time $t>0$. Hence, we have \eqref{align}.

We can now establish quantitative bounds on the growth of each oscillator. The bounds go just as in the absolute kernel, except rather than having $c=\inf h(r)$ determined by the initial conditions, we have $m=\min_{t,k,j} h(|x_k-x_j|)$.
\begin{align}
\l_+(t)\leq \l_+(0)\exp\left(\frac{k}{2m\mu}|\th_+-1|(0)\right).
\end{align}

Therefore each oscillator direction synchronizes exponentially fast to the direction of the order parameter, and
\begin{align}
    \lim_{t\to\infty} \|\zeta\|_2=\frac{1}{N}\sum_{j=1}^N\bar{\l}_j<\infty.
\end{align}

\end{proof}
This concludes the analysis of Model 1.

\section{Model 2: A generalization of the Schr\"odinger-Lohe Model}\label{sec:m2}
The second variation we consider is given by
\begin{align}\label{model2}
i\partial_t \psi_j&=H\psi_j+\frac{ik}{2N}\sum_{l=1}^N\left(\th_j\psi_l-\langle \psi_l,\psi_j\rangle \psi_j\right),\nonumber\\
\psi_j(x,0)&=\psi_j^0(x), \ \ j=1,\dots,N,\\
\dot{\th_j}&=\frac{\mu}{N}\sum_{k=1}^Nh(|x_j-x_k|)(\th_k-\th_j), \ \ \mu \geq 0, \nonumber\\
x_j&=\frac{1}{\|\psi_j\|_2^2}\int_{\mathbb{R}^d} x|\psi_j(x,t)|^2 \dx, \nonumber
\end{align}

Then for the order parameter $\zeta=\frac{1}{N}\sum_{l=1}^N\psi_l$, the equation becomes
\begin{align}
i\partial_t \psi_j=H\psi_j+\frac{ik}{2}\left(\th_j\zeta-\langle \zeta,\psi_j\rangle \psi_j\right).
\end{align}
and further the equation for $\zeta$ itself is
\begin{align}
i\partial_t \zeta=H\zeta+\frac{ik}{2}\left(\bar{\th}\zeta-\frac{1}{N}\sum_{j=1}^N\langle \zeta,\psi_j\rangle \psi_j\right),
\end{align}
where $\bar{\th}=\frac{1}{N}\sum_{j=1}^N\th_j$. The first of two important distinctions from the standard Schrodinger-Lohe model that should be made is the absence of normalization, so that
\begin{align}
\ddt \|\psi_j\|_2^2=k\Re\langle \zeta, \psi_j \rangle(\th_j-\|\psi_j\|_2^2),
\end{align}
thus the mass of each oscillator is not conserved, but rather, through the synchronization process, i.e., when $\Re\langle \zeta, \psi_j \rangle>0$, the square of the $L^2$ norm of $\psi_j$ approaches the value $\th_j$. Indeed, this is why in Proposition \ref{prop:gwp2}, we restrict ourselves to two sets of potential initial configurations. The presence of the parameters, $\th_j$, are the second distinction from the usual Schr\"odinger-Lohe model. Each $\th_j(t)>0$ represents an intrinsic desired mass for a given oscillator, and their motion will be governed by the Cucker-Smale dynamics,
\begin{align}
\dot{\th}_j&=\frac{\mu}{N}\sum_{l=1}^Nh(|x_l-x_j|)(\th_l-\th_j), \ \ \mu\geq 0,\\
x_j(t)&=\frac{1}{\|\psi_j\|_2^2}\int_{\R^d} x|\psi_j(x,t)|^2 \dx, \nonumber
\end{align}
with $h$ a radially symmetric heavy tail kernel, $\mu \geq 0$ a measure of coupling strength, where the case of $\mu=0$ represents each $\th_j$ remaining fixed, and $x_j$ the center of mass of an oscillator $\psi_j$.
Note that for $\th_j \equiv 1$ and $\|\psi_j^0\|_2=1$, for all $j$, the system reduces to the standard Schr\"odinger-Lohe model for quantum synchronization.

Again, since the mass of each oscillator is no longer conserved, it is useful to break each oscillator up into a part that measures the mass, and one that measures its direction. 
\begin{align}
\psi_j(x,t)&=\lambda_j(t) \phi_j(x,t),\\
\lambda_j(t)&=\|\psi_j(t)\|_2.  \nonumber
\end{align}
Then $\ddt \|\phi_j(x,t)\|_2=0$, and $\|\phi_j(x,t)\|_2\equiv 1$. Then the equations for $\phi_j$ and $\lambda_j$ are given by,
\begin{equation}\label{eq:m2new}\begin{aligned}
\dot\lambda_j=&\,\frac{k}{2}\Re\langle \zeta, \phi_j \rangle (\th_j-\lambda_j^2),\\
i\partial_t\phi_j=&\,\left(H+\frac{k}{2}\l_j\Im\langle\zeta,\phi_j\rangle\right)\phi_j+i\frac{k\th_j}{2\l_j}\left(\zeta-\Re\langle\zeta,\phi_j\rangle\phi_j\right).
\end{aligned}\end{equation}

Let us state the main theorem for this section:

\begin{theorem}\label{m2synch}
Global solutions to \eqref{model2} in the classes and with initial configurations as defined in Proposition \ref{prop:gwp2} undergo either exponential full synchronization (guaranteed for initial data satisfying the wedge condition \eqref{e:wedge}) or unstable synchronization to a bipolar state (only possible if $\|\psi_{j0}\|_2^2=\th_j$ for some $j$).
\end{theorem}

We split the analysis of this model into two cases: first with the wedge condition \eqref{e:wedge}.

\subsection{Synchronization on a wedge}
For Model 2, the synchronization results that we can prove are identical for both the absolute and heavy tail kernels, as it is no longer necessary that $\th_j \to 1$ in order to control the mass of each oscillator. Indeed, the Cucker-Smale dynamics can be turned on or off in this section ($\mu \geq 0$). 
\begin{proposition}\label{mod2sync}
The system \eqref{model2} with $\mu\geq 0$ and $h(r)$ satisfying the heavy tail condition \eqref{HT} undergoes exponential synchronization for initial conditions satisfying the wedge condition \eqref{e:wedge},
\begin{align}
    1-\Re\langle \phi_j,\phi_k\rangle\lesssim e^{-Ct}, \ \ \forall j,k.
\end{align}
If $\mu>0$ We further have $\th_j \to \bar{\th}=1$, $\l_j \to 1$, exponentially fast, as well as aggregation of centers of mass, $x_j  \to \frac{1}{\rho_\infty}\int_{\mathbb{R}^d}x|\zeta|^2 \dx$.
\end{proposition}
First we note that \eqref{e:wedge} is preserved in time.  Indeed, letting $\phi_+,\phi_-$ denote the two oscillator directions that minimize $\Re\langle \phi_j,\phi_k\rangle$ we see, utilizing Rademacher's lemma to differentiate,
\begin{align}
\ddt \Re \langle \phi_+,\phi_-\rangle&=\frac{k}{2}\frac{\th_+}{\l_+}(\Re\langle \zeta,  \phi_-\rangle-\Re\langle \zeta,  \phi_+\rangle\Re\langle \phi_+,  \phi_-\rangle)+\frac{k}{2}\lambda_+\Im\langle \phi_+,\zeta\rangle \Im\langle \phi_+,\phi_-\rangle\\
&+\frac{k}{2}\frac{\th_-}{\l_-}(\Re\langle \zeta,  \phi_+\rangle-\Re\langle \zeta,  \phi_-\rangle\Re\langle \phi_+,  \phi_-\rangle)+\frac{k}{2}\lambda_-\Im\langle \zeta, \phi_-\rangle \Im\langle \phi_+,\phi_-\rangle
\end{align}

Now, since $\Re\langle \phi_j,\phi_k\rangle\geq0$ for all $j,k$, $\Re\langle \zeta,\phi_j\rangle\geq c>0$, for all $j$ as well. Further, since $\phi_{\pm}$ minimizes $\Re\langle \phi_j,\phi_k\rangle,$ and without loss of generality, assuming $\Im\langle \phi_+,\zeta\rangle\geq0,$ then $\Im\langle \phi_+,\phi_-\rangle,\Im\langle \zeta,\phi_-\rangle\geq0.$. Therefore we also have the following inequalities,
\begin{align}
\|\zeta\|_2\Re\langle \phi_+,\phi_-\rangle\leq \Re\langle \zeta,\phi_{\pm}\rangle\leq \|\zeta\|_2
\end{align}

Hence,
\begin{align}
\ddt \Re \langle \phi_+,\phi_-\rangle\geq 0,
\end{align}
preserving \eqref{e:wedge}, and further,
\begin{align}
\Re\langle \phi_+,\phi_-\rangle\leq 1,
\end{align}
Implying, $\Re\langle \phi_+,\phi_-\rangle \to 1$ since $\ddt\Re \langle \phi_+,\phi_-\rangle > 0$, unless $\Re\langle \phi_+,\phi_-\rangle =\pm 1$.\\

Therefore, $\phi_j \to \frac{\zeta}{\|\zeta\|}$ for all $j$, which in turn implies aggregation of the centers of mass, $x_j \to \frac{1}{\rho_\infty}\int_{\mathbb{R}^d}x|\zeta|^2 \dx$. In particular this means that for $\mu>0$ the influence function $h(r)$ in the Cucker-Smale dynamics is bounded from below,  $\min_{t,j,k} h(|x_j-x_k|)=m>0.$ Hence, exponential alignment of each $\th_j$ occurs, as in \eqref{align}.

Continuing, to see that $\l_j \to 1$, we note  that since $\Re\langle \phi_+,\phi_-\rangle \to 1$, we have $\Re\langle \zeta,\phi_j\rangle(t)\geq c>0$ for all $t\geq T$. Thus,
\begin{align}
\ddt (\th_j-\l_j^2)&=\frac{\mu}{N}\sum_{k=1}^Nh(|x_k-x_j|)(\th_k-\th_j)-k\l_j\Re\langle \zeta,\phi_j\rangle(\th_j-\l_j^2)\\
&\leq \max_j|\th_j-\bar{\th}|(0)e^{-\mu mt}-c(\th_j-\l_j^2),
\end{align}
and Duhamel's Principle lets us conclude $\th_j-\l_j^2\to 0$ exponentially fast, and hence $\l_j^2 \to \bar{\th}=1$ as well.\\

Lastly,  we can now bootstrap the rate of synchronization of oscillatory directions to an exponential rate. Indeed, as $\Re\langle \phi_+,\phi_-\rangle\geq 0$ for all time, and with $\l_j \to 1, \th_j \to 1$ exponentially fast,
\begin{align}
\ddt(1-\Re\langle \phi_+,\phi_-\rangle)=&-k\frac{\th_+}{\l_+}(\Re\langle \zeta,  \phi_-\rangle-\Re\langle \zeta,  \phi_+\rangle\Re\langle \phi_+,  \phi_-\rangle)-k\lambda_+\Im\langle \phi_+,\zeta\rangle \Im\langle \phi_+,\phi_-\rangle\\
&-k\frac{\th_-}{\l_-}(\Re\langle \zeta,  \phi_+\rangle-\Re\langle \zeta,  \phi_-\rangle\Re\langle \phi_+,  \phi_-\rangle)-k\lambda_-\Im\langle \zeta, \phi_-\rangle \Im\langle \phi_+,\phi_-\rangle\\
&\leq -k(\Re\langle \zeta,\phi_+\rangle+\Re\langle \zeta,\phi_-\rangle)(1-\Re\langle \phi_+,\phi_-\rangle)\\
&+k(\Re\langle \zeta,  \phi_-\rangle-\Re\langle \zeta,  \phi_+\rangle\Re\langle \phi_+,  \phi_-\rangle)(1-\frac{\th_+}{\l_+})\\
&+k(\Re\langle \zeta,  \phi_+\rangle-\Re\langle \zeta,  \phi_-\rangle\Re\langle \phi_+,  \phi_-\rangle)(1-\frac{\th_-}{\l_-})
\end{align}
Now as $(1-\frac{\th_+}{\l_+})$ and $(1-\frac{\th_-}{\l_-})$ both converge to zero exponentially fast, and $\Re\langle \zeta,\phi_{\pm}\rangle\geq c>0$ for all time, we conclude by Duhamel's principle that $\Re\langle \phi_+,\phi_-\rangle \to 1$ also at an exponential rate.\\

Then for $\mu=0$. Let $\phi_{\pm}$ represent the the two oscillators such that $\min_{jk}\Re\langle \phi_j,\phi_k\rangle$ is achieved. Then, (using Rademacher's lemma)
\begin{align}
\ddt \Re \langle \phi_+,\phi_-\rangle&=\frac{k}{2}\frac{\th_+}{\l_+}(\Re\langle \zeta,  \phi_-\rangle-\Re\langle \zeta,  \phi_+\rangle\Re\langle \phi_+,  \phi_-\rangle)+\frac{k}{2}\lambda_+\Im\langle \phi_+,\zeta \rangle \Im\langle \phi_+,\phi_-\rangle\\
&+\frac{k}{2}\frac{\th_-}{\l_-}(\Re\langle \zeta,  \phi_+\rangle-\Re\langle \zeta,  \phi_-\rangle\Re\langle \phi_+,  \phi_-\rangle)+\frac{k}{2}\lambda_-\Im\langle \zeta, \phi_-\rangle \Im\langle \phi_+,\phi_-\rangle
\end{align}

Now, since $\Re\langle \phi_j,\phi_k\rangle\geq0$ for all $j,k$, $\Re\langle \zeta,\phi_j\rangle\geq c>0$, for all $j$ as well. Further, since $\phi_{\pm}$ minimizes $\Re\langle \phi_j,\phi_k\rangle,$ and without loss of generality, assuming $\Im\langle \phi_+,\zeta\rangle\geq0,$ then $\Im\langle \phi_+,\phi_-\rangle,\Im\langle \zeta,\phi_-\rangle\geq0.$ Therefore we also have the following inequalities,
\begin{align}
\|\zeta\|_2\Re\langle \phi_+,\phi_-\rangle\leq \Re\langle \zeta,\phi_{\pm}\rangle\leq \|\zeta\|_2
\end{align}

Hence,
\begin{align}
\ddt \Re \langle \phi_+,\phi_-\rangle&\geq 0,\\
\Re\langle \phi_+,\phi_-\rangle&\leq 1.
\end{align}
Therefore, $\Re\langle \phi_+,\phi_-\rangle \to 1$ and $\ddt\Re \langle \phi_+,\phi_-\rangle \to 0$.

Now, $\Re\langle \zeta,\phi_j\rangle\geq c>0$ for all $t>0$, and since
\begin{align}
\ddt(\th_j-\l^2_j)&=-k\Re\langle \zeta,\psi_j\rangle(\th_j-\l_j^2),
\end{align}
if $\th_j>\l_j^2$,
\begin{align}
\ddt(\th_j-\l^2_j)&\leq -kc\l_j(\th_j-\l_j^2).
\end{align}
Thus $\th_j-\l_j^2 \to 0$ exponentially fast. Now similarly if $\th_j<\l_j^2$,
\begin{align}
\ddt(\th_j-\l^2_j)&\geq -kc\l_j(\th_j-\l_j^2).
\end{align}
and again $\th_j-\l_j^2 \to 0$ exponentially fast.\\

To see that synchronization occurs at an exponential rate, we add and subtract appropriately, and using \eqref{e:wedge},
\begin{align}
\ddt \Re\langle \phi_+,\phi_-\rangle&=\frac{k\th_+}{2\l_+}(\Re\langle \zeta,\phi_-\rangle-\Re\langle \zeta,\phi_+\rangle \Re\langle \phi_+,\phi_-\rangle)+\frac{k}{2}\l_+\Im \langle \phi_+,\zeta\rangle\Im\langle\phi_+,\phi_-\rangle \nonumber\\
&+\frac{k\th_-}{2\l_-}(\Re\langle \zeta,\phi_+\rangle-\Re\langle \zeta,\phi_-\rangle \Re\langle \phi_+,\phi_- \rangle)+\frac{k}{2}\l_-\Im \langle \zeta,\phi_-\rangle\Im\langle\phi_+,\phi_-\rangle\nonumber\\
&\geq \frac{k}{2}(1-\Re\langle \phi_+,\phi_-\rangle)(\frac{\th_+}{\l_+}\Re\langle \zeta,\phi_-\rangle+\frac{\th_-}{\l_-}\Re\langle \zeta, \phi_+\rangle)\\
&+\frac{k}{2}\Re\langle \phi_+,\phi_-\rangle(\Re\langle \zeta,\phi_-\rangle-\Re\langle\zeta,\phi_+\rangle)(\frac{\th_+}{\l_+}-\frac{\th_-}{\l_-})\nonumber\\
\end{align}
If $\frac{\th_+}{\l_+}\geq \frac{\th_-}{\l_-}$, then,
\begin{align}
\ddt \Re\langle \phi_+,\phi_-\rangle&\geq \frac{k}{2}\|\zeta\|(1-\Re\langle \phi_+,\phi_-\rangle)(\frac{\th_-}{\l_-}(\Re\langle \phi_+,\phi_-\rangle+\Re\langle\frac{\zeta}{\|\zeta\|},\phi_+\rangle) +\frac{\th_+}{\l_+}(\Re\langle \frac{\zeta}{\|\zeta\|}, \phi_-\rangle-\Re\langle \phi_+,\phi_-\rangle))\nonumber\\
&\geq\frac{k}{2}\|\zeta\|(1-\Re\langle \phi_+,\phi_-\rangle)(\frac{\th_-}{\l_-}(\Re\langle \phi_+,\phi_-\rangle+\Re\langle\frac{\zeta}{\|\zeta\|},\phi_+\rangle))
\end{align}
and if $\frac{\th_+}{\l_+}<\frac{\th_-}{\l_-}$, then
\begin{align}
\ddt \Re\langle \phi_+,\phi_-\rangle&\geq \frac{k}{2}\|\zeta\|(1-\Re\langle \phi_+,\phi_-\rangle)(\frac{\th_+}{\l_+}(\Re\langle \phi_+,\phi_-\rangle+\Re\langle\frac{\zeta}{\|\zeta\|},\phi_-\rangle) +\frac{\th_-}{\l_-}(\Re\langle \frac{\zeta}{\|\zeta\|}, \phi_+\rangle-\Re\langle \phi_+,\phi_-\rangle))\nonumber\\
&\geq\frac{k}{2}\|\zeta\|(1-\Re\langle \phi_+,\phi_-\rangle)(\frac{\th_+}{\l_+}(\Re\langle \phi_+,\phi_-\rangle+\Re\langle\frac{\zeta}{\|\zeta\|},\phi_-\rangle))
\end{align}
In either case we have,
\begin{align}
\ddt (&1-\Re\langle \phi_+,\phi_-\rangle)\leq -c(1-\Re\langle \phi_+,\phi_-\rangle),\nonumber\\
(&1-\Re\langle \phi_+,\phi_-\rangle)(t)\leq (1-\Re\langle \phi_+,\phi_-\rangle(0))e^{-ct}.
\end{align}
Thus proving the proposition.

\subsection{Synchronization with frozen parameters}
For Model 2 and $\mu=0$, turning the Cucke-Smale dynamics off, the situation is different.  We were restricted to only initial conditions satisfying \eqref{e:wedge} for the $\mu>0$ case, and while in the case where \eqref{e:wedge} holds with $\mu=0$, Proposition \ref{mod2sync} still holds,  we can also describe the synchronization dynamics for a wider class of initial conditions. An important note is that with the Cucker-Smale  dynamics effectively being turned off, each $\th_j$ remains fixed so that rather than $\th_j \to \bar{\th}$ and $\l_j^2\to \bar{\th}$, we see $\l_j^2 \to \th_j$ for each individual $j$ so that in the end each oscillator has a different mass,  determined by the initial desired mass $\th_j$. This can be appealing as we can see synchronization occur, while being able to choose the mass of each quantum oscillator.
\begin{proposition}\label{ncssynch}
The synchronization dynamics of global solutions  to system \eqref{model2}, with $\mu=0$, and initial conditions such that $\l_j^2(0)\leq \th_j$ for all $j=1,...,N$, either complete synchronization at an exponential rate occurs, or the system converges to an unstable Bipolar state.
\end{proposition}

\begin{proof}

Now, without a condition like \eqref{e:wedge} we resort to analyzing the order parameter again. We compute,

\begin{align}
\ddt \|\zeta\|_2^2= k\left(\bar{\th}\|\zeta\|_2^2-\frac{1}{N}\sum_{l=1}^N\lambda_l^2 \Re[\langle \phi_l, \zeta\rangle]^2\right)
\end{align}

Now since $\lambda_j^2(0)\leq \th_j$ and $\dot{\lambda}_j=\frac{k}{2} \Re\langle \zeta, \phi_j \rangle (\th_j-\lambda_j^2),$ we have $\lambda_j^2(t)\leq \th_j$ for all $t>0$ as well. Further,
\begin{align}
|\Re[\langle \phi_l, \zeta\rangle]^2|\leq |\langle \phi_l, \zeta\rangle|^2\leq \|\zeta\|_2^2,
\end{align}
since $\|\phi_l\|_2=1$. Therefore, 
\begin{align}
\frac{1}{N}\sum_{l=1}^N\lambda_l^2 \Re[\langle \phi_l, \zeta\rangle]^2\leq \bar{\th}\|\zeta\|_2^2
\end{align}
and hence $\ddt \|\zeta\|_2^2\geq 0$. Now $\|\zeta\|_2$ is monotonic, and to see it is bounded,
\begin{align}
\|\zeta\|_2=\left\|\frac{1}{N}\sum_{l=1}^N\lambda_l\phi_l\right\|_2\leq \frac{1}{N}\sum_{l=1}^N\lambda_l\leq \sqrt{\th_M},
\end{align}
where $\th_M=\max_j\th_j$.  Therefore $\|\zeta\|_2 \to \rho_{\infty}\in(0,\sqrt{\th_M}]$.

Further, we know $\ddt \|\zeta\|_2^2=\bar{\th}\|\zeta\|_2^2-\frac{1}{N}\sum_{l=1}^N\lambda_l^2 \Re[\langle \phi_l, \zeta\rangle]^2 \to 0$. Since $\lambda_l^2\leq \th_l$ and $\Re[\langle \phi_l, \zeta\rangle]^2\leq \|\zeta\|^2$,
\begin{align}
\bar{\th}\|\zeta\|_2^2-\frac{1}{N}\sum_{l=1}^N\lambda_l^2 \Re[\langle \phi_l, \zeta\rangle]^2&=\frac{1}{N}\sum_{l=1}^N\th_l\|\zeta\|_2^2-\lambda_l^2\Re[\langle \phi_l, \zeta\rangle]^2,\\
&=\frac{1}{N}\sum_{l=1}^N (\th_l-\lambda_l^2)\|\zeta\|_2^2+\frac{1}{N}\sum_{l=1}^N\lambda_l^2(\|\zeta\|_2^2-\Re[\langle \phi_l, \zeta\rangle]^2).
\end{align}
Now each of these sums are nonnegative, so $\ddt\|\zeta\|_2^2 \to 0$ if and only if $\lambda_l^2 \to \th_l$ and $\phi_l \to \pm\frac{1}{\|\zeta\|_2}\zeta$. Therefore we have synchronization.

For the special case of $\l_j^2(0)<\th_j$, we can further guarantee that the oscillator avoids the antipole.

Indeed, we suppose for this $j$ that $\phi_j \to -\frac{1}{\|\zeta\|_2}\zeta$. In that case there exists a time $T$ and constant $C<0$ such that $\frac{k}{2}\Re \langle \zeta,\phi_j\rangle\leq C<0$ for all $t\geq T$.  Then,
\begin{align}
\ddt\lambda_j&=\frac{k}{2}\Re\langle \zeta, \phi_j\rangle(\th_j-\lambda_j^2)\\
&\leq C(\th_j-\lambda_j^2), \ \ \text{on} \ [T,\infty),\\
\int_T^t \frac{d\lambda_j}{\th_j-\lambda_j^2(s)}&\leq \int_T^t C \dt,\\
\lambda_j(t)&\leq \sqrt{\th_j}\frac{(\sqrt{\th_j}+\lambda_j^T)e^{2\sqrt{\th_j}C(t-T)}-(\sqrt{\th_j}-\lambda_j^T)}{(\sqrt{\th_j}-\lambda_j^T)+(\sqrt{\th_j}+\lambda_j^T)e^{2\sqrt{\th_j}C(t-T)}}, \ \ \text{on} \ [T,\infty).
\end{align}

Therefore, at $t^*=\frac{1}{2\sqrt{\th_j}C}\ln\left(\frac{\sqrt{\th_j}-\lambda_j^T}{\sqrt{\th_j}+\lambda_j^T}\right)+T>T$, we have $\lambda_j(t^*)\leq 0$. In which case, either $\lambda_j(t^*)<0$, a contradiction, or $\lambda_j(t^*)=0$. If $\lambda_j(t^*)=0$, then $\dot{\lambda}_j(t^*)=\frac{k}{2}\Re\langle \zeta,\phi_j\rangle\th_j\leq C\th_j<0$, also a contradiction.

Therefore $\phi_j \to +\frac{1}{\|\zeta\|_2}\zeta$. Hence if for every $j=1,...,N$, the mass is strictly below its corresponding desired mass, $\l_j^2(0)<\th_j$, then complete synchronization at an exponential rate occurs.

Indeed, $\l_j^2 \to \th_j$ occurs at an exponential rate since $\Re\langle \zeta,\phi_j\rangle \geq C>0$. for all $t>T$ from fact that complete synchronization occurs. Further, we can conclude exponential synchronization as well, since complete synchronization occurs, there is a time $T$ such that for all $t>T$ \eqref{e:wedge} holds, and the argument in Proposition \ref{mod2sync} yields the exponential rate of synchronization.
\end{proof}

\section{Bipolar Synchronization and Incoherence}\label{sec:ex}
Let us now state two illustrative examples for Incoherence and Bipolar Synchronization behavior. These examples are important, as the standard Schr\"odinger-Lohe model, having identical oscillators, cannot possibly converge to an incoherent state due to the monotonicity (growth) of the order parameter. To construct an example for which bipolar synchronization emerges is not difficult, but it occurs in a somewhat contrived manner, as expected for an unstable equilibrium.
\begin{example}\label{bip}
First, Bipolar Synchronization can only occur if for some $j$, $\lambda_j^2(0)=\th_j$. Take any initial conditions for which $\l_1^2(0)=\th_1$, and $\phi_1(0)=-\frac{\zeta}{\|\zeta\|_2}(0)$, and $\l_j^2(0)<\th_j$ for all other $j$. Now symmetrizing around $\zeta$ so that we have a system of $2N$ oscillators where the first $N$ are the original oscillators and the oscillators corresponding to $N+1,...,2N$ satisfy, for each $j=1,...,N$,
\begin{align}
\Re\left\langle\phi_j,\frac{\zeta}{\|\zeta\|_2}\right\rangle(0)&=\Re\left\langle \phi_{j+N},\frac{\zeta}{\|\zeta\|_2}\right\rangle(0),\\
\Im\left\langle\phi_j,\frac{\zeta}{\|\zeta\|_2}\right\rangle(0)&=-\Im\left\langle \phi_{j+N},\frac{\zeta}{\|\zeta\|_2}\right\rangle(0),\\
\l_j(0)&=\l_{j+N}(0),\\
\th_j&=\th_{j+N}.
\end{align}
Then this system converges to the bipolar state.

Indeed, since $\l_k^2\leq \th_k$ for all $k$, Proposition \ref{prop:gwp2} guarantees a global solution and Proposition \ref{ncssynch} guarantees $\phi_j \to \pm \frac{\zeta}{\|\zeta\|_2}$. To see that $\phi_1,\phi_{N+1}=-\frac{\zeta}{\|\zeta\|_2}$ for all time we compute,
\begin{align}
\ddt \left\langle \frac{\zeta}{\|\zeta\|_2}, \phi_1\right\rangle|_{t=0}&=\frac{k}{2N}\sum_{l=1}^{2N}\l_l^2\left[\left\langle \phi_l,\frac{\zeta}{\|\zeta\|_2}\right\rangle^2-\Re\left[\left\langle\phi_l,\frac{\zeta}{\|\zeta\|_2}\right\rangle^2\right]\right],\\
\ddt \Re\left\langle \frac{\zeta}{\|\zeta\|_2}, \phi_1\right\rangle|_{t=0}&=0,\\
\ddt \Im\left\langle \frac{\zeta}{\|\zeta\|_2}, \phi_1\right\rangle|_{t=0}&=\frac{k}{2N}\sum_{l=1}^{2N}\l_l^2\Im\left[\left\langle \phi_l,\frac{\zeta}{\|\zeta\|_2}\right\rangle^2\right],\\
&=\frac{k}{N}\sum_{l=1}^N\l_l^2\Re\left\langle\phi_l,\frac{\zeta}{\|\zeta\|_2}\right\rangle\Im\left\langle\phi_l,\frac{\zeta}{\|\zeta\|_2}\right\rangle+\frac{k}{N}\sum_{l=N+1}^{2N}\l_l^2\Re\left\langle\phi_l,\frac{\zeta}{\|\zeta\|_2}\right\rangle\Im\left\langle\phi_l,\frac{\zeta}{\|\zeta\|_2}\right\rangle,\\
&=0.
\end{align}
Therefore  $\phi_1=-\frac{\zeta}{\|\zeta\|_2}$ for all time, and the same for $\phi_{N+1}$, resulting in a bipolar synchronization.

\end{example}
Incoherence, $\|\zeta\|_2 \to 0$, can happen for a similar initial configuration:
\begin{example}\label{inc}
For the same system as the above example except where $\l_1^2(0)>\th_1$, we see $\|\zeta\|_2 \to 0$.

Indeed, $\phi_1(t)=-\zeta(t)/\|\zeta\|_2$ for all time, but since $\dot{\l}_1=-\frac{k}{2}\|\zeta\|_2(t)(\th_1-\l_1^2)$, we see that $\l_1^2$ grows so that 
we expect $\lim_{t\to\infty}\|\zeta\|_2=0$, resulting in an incoherent state. We note however that Proposition \ref{prop:gwp2} no longer applies, so we don't automatically get global solutions or synchronization of the other oscillators. However, if the remaining $2N-2$ oscillators satisfy the wedge condition \eqref{e:wedge}, then we still have
\begin{align}
&\ddt \Re \langle \phi_+,\phi_-\rangle\geq 0,\\
&\Re\langle \phi_{\pm},\zeta\rangle \geq 0,
\end{align}
where $\phi_{+},\phi_-$ are chosen from these remaining oscillators. Therefore, the remaining $N-2$ oscillators must all remain bounded.

To see that $\|\psi_1\|_2$ also remains bounded so that we can extend from a local  to a global solution, we look at the order parameter again. Indeed if $\|\psi_1\|_2\to \infty$ as $t \to T<\infty$, then since all the other oscillators remain bounded, $\|\zeta\|_2 \to \infty$  as $t \to T$ as well.  However, for $\|\psi_1\|_2$ large, we see that $\ddt \|\zeta\|_2<0$,  a  contradiction.  Therefore this initial configuration also admits a global solution so that $\|\zeta\|_2 \to 0$.

Although, with $\|\zeta\|_2 \to 0$, it is possible that $\Re\langle \phi_+,\phi_-\rangle \to c<1$, which would truly represent an incoherent state, different from the bipolar synchronization. Both these examples take advantage of the fact that the system is symmetry preserving with respect to the order parameter.
\end{example}

\end{document}